\documentclass[12pt,a4paper]{amsart}
\usepackage{amscd}
\usepackage{amssymb}
\usepackage[centering,text={15.5cm,22cm}]{geometry}
\usepackage{float} %to be able to have figure in minipage
\usepackage{graphicx,color,caption}
\usepackage[all]{xy}
\SelectTips{cm}{11}
\usepackage{mathrsfs}
\usepackage{marvosym}
\usepackage{stmaryrd}
\usepackage{srcltx}
\usepackage{hyperref}
\usepackage{upgreek} %for \upnu

\definecolor{shadecolor}{rgb}{1,0.9,0.7}

\newtheorem{theorem}{Theorem}  %[section]
\newtheorem{lemma}[theorem]{Lemma}
\newtheorem{lemma-definition}[theorem]{Lemma-Definition}
\newtheorem{proposition}[theorem]{Proposition}
\newtheorem{corollary}[theorem]{Corollary}

\theoremstyle{definition}
\newtheorem{setup}[theorem]{Setup}
\newtheorem{definition}[theorem]{Definition}
\newtheorem{construction}[theorem]{Construction}

\newtheorem{example}[theorem]{Example}
\newtheorem{cexample}[theorem]{Counter-Example}

\theoremstyle{remark}
\newtheorem{remark}[theorem]{Remark}

\numberwithin{equation}{section}
\numberwithin{figure}{section}

\newcommand {\lfor} {\llbracket}
\newcommand {\rfor} {\rrbracket}

\newcommand{\ZZ} {\mathbb{Z}}
\newcommand{\QQ} {\mathbb{Q}}
\newcommand{\RR} {\mathbb{R}}
\newcommand{\CC} {\mathbb{C}}

\newcommand{\PP} {\mathbb{P}}

\newcommand {\shA}  {\mathcal{A}}

\newcommand {\shB}  {\mathcal{B}}

\newcommand {\shF}  {\mathcal{F}}
\newcommand {\shG}  {\mathcal{G}}

\newcommand {\shO}  {\mathcal{O}}

\newcommand {\shT}  {\mathcal{T}}

\newcommand {\shP}  {\mathcal{P}}

\newcommand {\foX}  {\mathfrak{X}}

\newcommand {\fou}  {\mathfrak{u}}

\newcommand {\bary} {{\operatorname{bar}}}

\newcommand {\codim} {\operatorname{codim}}

\newcommand {\eps}  {\varepsilon}

\newcommand {\GL}  {\operatorname{GL}}

\newcommand {\Gr}  {\operatorname{Gr}}

\newcommand {\Hom}  {\operatorname{Hom}}

\newcommand {\hra} {\hookrightarrow}

\newcommand {\id}  {\operatorname{id}}
\newcommand {\im}  {\operatorname{im}}

\newcommand {\Int}  {\operatorname{Int}}

\renewcommand {\ker } {\operatorname{ker}}

\newcommand {\lra}  {\longrightarrow}

\renewcommand {\max} {{\operatorname{max}}}

\let\op\operatorname

\newcommand {\orr}  {\operatorname{or}}
\newcommand {\ori} {\operatorname{or}}

\renewcommand{\P}  {\mathscr{P}}

\newcommand {\ra}  {\to}

\newcommand {\rk} {\operatorname{rk}}

\newcommand {\Sing} {\operatorname{Sing}}

\newcommand {\Spf}  {\operatorname{Spf}}
\newcommand {\sra} {\twoheadrightarrow}

\newcommand {\triang} {\triangle}

\def\mydate{\ifcase\month \or January\or February\or March\or
April\or May\or June\or July\or August\or September\or October\or 
November\or December\fi \space\number\day,\space\number\year}

\begin{document}

%===========================================================

\begin{abstract} 
We introduce a cap product pairing for homology and cohomology of tropical cycles on integral affine manifolds with singularities. We show the pairing is perfect over $\QQ$ in degree one when the manifold has at worst symple singularities. By joint work with Siebert, the pairing computes period integrals and its perfectness implies the versality of canonical Calabi--Yau degenerations. We also give an intersection theoretic application for Strominger--Yau--Zaslow fibrations. The treatment of the cap product and Poincar\'e--Lefschetz by simplicial methods for constructible sheaves might be of independent interest.
\end{abstract}

\title
[A homology theory for tropical cycles and a perfect pairing]
{A homology theory for tropical cycles on integral affine manifolds and a perfect pairing}
%\mbox{\tiny -preliminary version-}}
\author{Helge Ruddat}

\address{\tiny JGU Mainz, Staudingerweg 9, 55128 Mainz \quad\&\quad Universit\"at Hamburg, Bundesstr. 55, 20146 Hamburg}
\email{ruddat@uni-mainz.de, helge.ruddat@uni-hamburg.de}

%\address{\tiny Universit\"at Hamburg, Institut f\"ur Mathematik, Bundesstr. 55, 20146 Hamburg, Germany}
%\email{helge.ruddat@uni-hamburg.de}

\thanks{This work was supported by a Carl--Zeiss Postdoctoral Fellowship and DFG grant RU 1629/4-1}

\maketitle
\setcounter{tocdepth}{1}
\tableofcontents
\bigskip

%===========================================================
%===========================================================
\section*{Introduction}
The work on this article started out with the intention to develop a homology theory for the type of cycles given in \cite[Definition 7.2]{CBM13} which then found a use in \cite{RS20}. 
Let us begin with a simpler example.
\subsection{A fibred K3 as a leading example}
Let $f\colon X\ra B:=\PP^1$ be an elliptic fibration of a K3 surface. The topology of the cup product $H^2(X,\ZZ)\otimes H^2(X,\ZZ)\ra\ZZ$ can be studied using the Leray filtration of $f$. 
In the case where we have a section $B\ra X$ and all fibers are irreducible, i.e.~$R^2f_*\ZZ\cong\ZZ$, the section determines uniquely an isomorphism
$$H^2(X,\ZZ)= H^2(B,\ZZ)\oplus H^1(B,R^1f_*\ZZ)\oplus H^0(B,\ZZ)$$
identifying the middle summand as the orthogonal complement of the outer two. The outer two summands form a hyperbolic plane, that is, their sum is isomorphic to $H=\ZZ^2$ with pairing given by $\bigl( \begin{smallmatrix}0 & 1\\ 1 & 0\end{smallmatrix}\bigr)$. The cup product of the middle summand is the natural pairing
$$H^1(B,R^1f_*\ZZ)\otimes H^1(B,R^1f_*\ZZ)\ra H^2(B,R^2f_*\ZZ)=\ZZ.$$
The lattice $\ZZ^8$ with pairing given by $(-1)$ times the $E_8$ matrix is called $-E_8$. 
As lattices, there exists an isomorphism $H^2(X,\ZZ)\cong-E_8^{\oplus 2}\oplus H^{\oplus 3}$, and in fact $H^1(B,R^1f_*\ZZ)\cong -E_8^{\oplus 2}\oplus H^{\oplus 2}$. 
(The latter statement can be shown using \cite{Sy01} as was pointed out in \cite[\S1.4.3]{RS20})

\subsection{Frames on affine manifolds}
If $\Delta\subset B$ denotes the discriminant of $f$, then $\check\Lambda:=(R^1f_*\ZZ)|_{B\setminus\Delta}$ is a local system on $B\setminus\Delta$ with stalks isomorphic to $\ZZ^2$.
Assume next that all singular fibers of $f$ are $I_1$-fibers, i.e.~rational curves with a single node, then,
for $\iota\colon B\setminus\Delta\hra B$ the inclusion of the complement, the natural map
\begin{equation} \label{eq-simple-iso}
R^1f_*\ZZ\ra \iota_*\iota^{-1}R^1f_*\ZZ=\iota_*\check\Lambda
\end{equation}
is an isomorphism. This property for a fibration was called $\ZZ$-{\bf simple} in \cite{topMS}.
We therefore have a unimodular bilinear form 
\begin{equation} \label{eq-bil-form}
H^1(B,\iota_*\check\Lambda)\otimes H^1(B,\iota_*\check\Lambda)\ra\ZZ.
\end{equation} 
Since the regular fibers of $f$ satisfy Poincar\'e duality and are of dimension two, we furthermore obtain an identification $\Lambda\cong\check\Lambda$ where $\Lambda:=\Hom(\check\Lambda,\ZZ)$.
A proof of the following form of Poincar\'e--Lefschetz duality is included in Appendix~\ref{sec-hocoho-compare}. 
Let $H_k$ refer to the $k$th {\bf singular sheaf homology}, see e.g.~\cite[VI-\S12, p.443]{Br97} (the notation there is ${}_\Delta H^c_k$).
\begin{theorem}
[Poincar\'e--Lefschetz, Theorem~\ref{cor-iso-ho-coho}, cf.~Theorem\,12.1.3 in \cite{Cu13}]
\label{thm-lefschetz}
For a constructible sheaf $\shF$ on a compact oriented PL topological $n$-manifold $B$, there are natural isomorphisms $H_k(B,\partial B;\shF)\stackrel\sim\ra H^{n-k}(B,\shF)$ and
$H_k(B,\shF)\stackrel\sim\ra H^{n-k}(B,\partial B;\shF)$ for each $k$.
\end{theorem}
In our example, the theorem gives a canonical isomorphism
$H_1(B,\iota_*\Lambda) = H^1(B,\iota_*\Lambda)$.
By Theorem~\ref{thm-pairings-agree} below, the isomorphism in Theorem~\ref{thm-lefschetz} is such that \eqref{eq-bil-form} re-identifies as the cap product pairing
\begin{equation}\label{main-pairing}
H_1(B,\iota_*\Lambda)\otimes H^1(B,\iota_*\check\Lambda)\ra\ZZ
\end{equation}
given by pairing $\check\Lambda$ with $\Lambda$ and homology with cohomology, see Lemma~\ref{lemma-cap-zero}.
The appeal of \eqref{main-pairing} is that it is valid as stated also in higher dimensions and other degrees, see~\eqref{eq-gen-pairing} below.

To make even better use of \eqref{main-pairing}, we are going to make $\Lambda$ intrinsic to $B$. 
In our motivating example, Hyperk\"ahler rotation turns $f\colon X\ra B$ into a Lagrangian fibration. 
By the Liouville–Arnold theorem and the notion of action-angle coordinates \cite{Ar78,D}, integrating the symplectic form over the cylinders swept out by a pair of generators of $\Lambda_b=H_1(f^{-1}(b),\ZZ)$ as we move $b\in B\setminus \Delta$ gives a system of integral affine coordinates on $B\setminus \Delta$. In other words, we obtain an embedding of $\Lambda$ as a frame in the tangent sheaf $\shT_{B\setminus \Delta}$. Likewise, $\check\Lambda$ frames $\shT^*_{B\setminus \Delta}$.
The 2-torus-bundle $\shT^*_{B\setminus \Delta}/\check \Lambda$ over $B\setminus \Delta$ is symplectomorphic to the restriction of $X\ra B$ to $B\setminus \Delta$, up to twisting by the Chern class of the bundle, see~\cite{D}. 
Hence, the symplectic topology of $X\ra B$ away from the discriminant can be reconstructed from knowing the affine structure on ${B\setminus \Delta}$.
With some control/restriction on the singularities of the affine structure along $\Delta$, the entire topology of $X\ra B$ can be constructed from $B$, \cite{topMS,CBM09,Pr18,Pr19}.
We will therefore study the pairing \eqref{main-pairing} in the setting of spaces given by the following definition.

\begin{definition} \label{def-affmfd}
An {\bf integral affine manifold with singularities} is a PL topological manifold $B$ together with an atlas on a dense open set with transition maps in $\GL_n(\ZZ)\ltimes\RR^n$. Let $\Delta$ denote the complement of the open set and $\iota\colon B\setminus\Delta\hra B$ the inclusion. 
We require that the pair $(B,\Delta)$ is locally PL homeomorphic to $(\RR^n,\Sigma)$ for $\Sigma$ a finite fan of polyhedral cones with support in codimension 2.

We obtain a local system of integral tangent vectors $\Lambda\subset \shT_{B\setminus \Delta}$ on $B\setminus\Delta$ by picking a full lattice in a single stalk and then parallel\footnote{We remark here that we use the ordinary flat linear connection. An affine manifold also has a flat affine connection by pulling back the affine connection via charts from affine space, see~\cite[\S1]{logmirror1} or \cite{GH84}.} translating it around. We denote its dual by $\check\Lambda=\Hom(\Lambda,\ZZ)$. 
\end{definition}
The main purpose of this article is to make a statement under which conditions the pairing \eqref{main-pairing} is perfect. This turns out to be rather sensitive to the type of singularities along $\Delta$, see Counter-Example~\ref{counterexample-conifold}.
We already motivated the notion of {\bf symple}\footnote{We say \emph{symple} for the distinction to the similar notion of \emph{simple} given in \cite[Definition\,1.60]{logmirror1}. Most importantly: \emph{simple implies symple}.} singularities in \eqref{eq-simple-iso} and a general definition is given in \S\ref{sec-simple} below. 
Roughly speaking, locally $\Delta$ needs to be a transverse union of products of tropical hyperplanes.
{\bf Focus-focus points} \cite{Wi},\cite{D},\cite[Definition\,4.2]{Sy01}, see Figure~\ref{fig-focusfocus}, are examples of symple singularities. 
Our main result is the following.
\begin{theorem}[Theorem~\ref{pairingthm} below] \label{main-thm-intro}
If $B$ is an integral affine manifold with symple singularities, then the pairing \eqref{main-pairing} is perfect over $\QQ$.
\end{theorem}
The theorem \emph{does not} just follow from linear duality, e.g.~\cite[\S12.2.1]{Cu13}, because of the non-derived $\iota_*$ operation.

\subsection{Analyticity, versality and periods for canonical Calabi--Yau families}
We next explain how Theorem~\ref{main-thm-intro} enables the proof of the analyticity and versality of canonical Gross--Siebert Calabi--Yau families.
Recall from \cite{affinecomplex,theta} that, given an integral affine manifold with simple singularities $B$ and additionally a polyhedral decomposition $\shP$ with multivalued strictly convex function $\varphi$, Gross and Siebert associate to the triple $(B,\shP,\varphi)$ a canonical formal family 
\begin{equation} \label{GS-family}
\foX\ra\Spf\underbrace{\CC[H^1(B,\iota_*\check\Lambda)^*]}_{=:A}\lfor t\rfor
\end{equation} 
where $H^1(B,\iota_*\check\Lambda)^*:=\Hom(H^1(B,\iota_*\check\Lambda),\ZZ)$. The first Chern class of $\varphi$ is an element $c_1(\varphi)\in H^1(B,\iota_*\check\Lambda)$ and hence can be paired with a cycle $\beta\in H_1(B,\iota_*\Lambda)$ under \eqref{main-pairing} to give an integer $\langle c_1(\varphi),\beta\rangle$. 
Furthermore, \eqref{main-pairing} also gives a map $H_1(B,\iota_*\Lambda) \ra H^1(B,\iota_*\check\Lambda)^*,\beta\mapsto \beta^*$. 
\begin{theorem}[\cite{RS20}] \label{GR-theorem}
The (Laurent) monomial $z^{\beta^*}t^{\langle c_1(\varphi),\beta\rangle}\in A[t^{\pm 1}]$ obtained from $\beta\in H_1(B,\iota_*\Lambda)$ is an exponentiated period integral for the family \eqref{GS-family}, namely
$$z^{\beta^*}t^{\langle c_1(\varphi),\beta\rangle}=\exp\left(\frac1{(2\pi i)^{n-1}}\int_{r_{1,1}(\beta)}\Omega_{\foX/\Spf A\lfor t\rfor}\right)$$
noting that the map $r_{1,1}\colon H_1(B,\iota_*\Lambda)\ra H_n(X,\ZZ)/\im r_{0,0}$ is given in \eqref{eq-map-rpq} below. 
\end{theorem}
Taking the exponential deletes the ambiguity of adding elements in $\im r_{0,0}$. In view of the difference $X$ versus $\foX$, making sure the integral in Theorem~\ref{GR-theorem} is well-defined, even at finite $t$-order, is a non-trivial part of the statement of Theorem~\ref{GR-theorem}.
In this context, Theorem~\ref{main-thm-intro} ensures that the period monomials (with $\langle c_1(\varphi),\beta\rangle\ge 0$) in Theorem~\ref{GR-theorem} generate a codimension one subring of $A[t]$. Since the missing dimension is explained by a natural equivariant $\CC^*$-action on domain and co-domain of \eqref{GS-family}, using that periods are always holomorphic functions, it was proved in \cite{RS20} that \eqref{GS-family} is (locally in the base) the completion of an analytic family and furthermore log semi-universal. The families of the form \eqref{GS-family} appear as mirror symmetry duals, so the gain is that the mirror dual is not just a formal scheme but an honest complex manifold (or possibly an orbifold if $\dim B\ge 4$).
In the situation where $\eqref{GS-family}$ comes with an embedding in an ambient toric variety, Yamamoto gave a cup product interpretation of the period integral in Theorem~\ref{GR-theorem} and the general pairing \eqref{eq-gen-pairing} below in terms of the polarized log Hodge structure induced by the ambient cohomology, see \cite[Definition-Lemma 1.2 and Remark 7.2]{Ya}.

\subsection{Tropical cycles and ordinary cycles}
The construction of $r_{1,1}$ in the previous section goes back to Siebert's ETH talk\footnote{``Canonical coordinates in mirror symmetry and tropical disks'' at the conference \emph{Symplectic Geometry and Physics}, Z\"urich, Sep 5, 2007} and is motivated from the $\dim B=2$ case given by Symington \cite{Sy01}, see~Figure~\ref{fig-goggle-to-relative} below.
We call the homology groups 
$$H_{p,q}:=H_{q}(B,\iota_*\bigwedge^p\Lambda)$$
\emph{homology of tropical cycles} or \emph{affine homology} and we call representatives of $H_{p,q}$ \emph{tropical $p,q$-cycles} or if $p=q$ just \emph{tropical $p$-cycles}. 
The work \cite{CBM13} features tropical $1,2$-cycles (in a relative version). Denoting $H^{p,q}:=H^{q}(B,\iota_*\bigwedge^p\check\Lambda)$, we prove in Theorem~\ref{pairingthm} there is a natural pairing
\begin{equation} \label{eq-gen-pairing}
H_{p,q}\otimes H^{p,q} \ra\ZZ
\end{equation}
on an integral affine manifold with singularities which is compatible with the cap product pairing of dual local systems
$$H_q(B\setminus\Delta,\bigwedge^p\Lambda)\otimes H^q(B\setminus\Delta,\bigwedge^p\check\Lambda)\ra\ZZ.$$

For an affine manifold with singularities $B$ of dimension $n$, assume that $f\colon X\ra B$ is a compactification of the torus bundle $X^\circ:=\shT^*_{B\setminus \Delta}/\check \Lambda$ (by possibly singular fibers) 
so that $f$ permits a section.
Such compactifications are constructed in \cite{RZ20,RZ} where we also give a sequence of natural homomorphisms 
\begin{equation} \label{eq-map-rpq}
r_{p,q}\colon H_q(B,\iota_*\bigwedge^p\Lambda)\ra H_{n-p+q}(X,\ZZ)/\im r_{p-1,q-1}.
\end{equation}
These maps are given by an explicit construction of cycles. 
Setting $$W_{p+q}:=\im r_{p,q}+ \im r_{p-1,q-1}+ \im r_{p-2,q-2}+...,$$ we obtain an increasing filtration $W_\bullet$ of $H_{k}(X,\ZZ)$ for $k=n-p+q$. 
For $\dim B=3$, the Leray filtration of $f$ was shown to agree with $W_\bullet$ in \cite{Gr98}. We thus generalized the leading K3 example to arbitrary dimension.
In view of \cite[\S1.4.5]{RS20}, the perfectness of the pairing \eqref{main-pairing} is equivalent to saying that $r_{1,1}$ induces an isomorphism
$H_{1,1}\cong W_2/W_0$ over $\QQ$.

\subsection{Tropical (co)homology}
The homology of tropical cycles $H_{p,q}$ is naturally isomorphic to $H^{n-q}(B,\iota_*\bigwedge^p\Lambda)$ by Theorem~\ref{thm-lefschetz}. 
These latter cohomology groups are known by the name \emph{affine cohomology}. 
They have an important application in the theory of toric log Calabi-Yau spaces as follows. 
A toric log Calabi-Yau space $X$ is a certain type of maximally degenerate Calabi-Yau variety introduced in \cite{logmirror1}. 
It appears as the central fibre in \eqref{GS-family}.
These spaces come with natural cohomology groups, the log Hodge groups $H^q(X,\Omega^p)$ that were introduced in \cite{logmirror2}. It was proved in 
\cite{logmirror2,Ru} that there is a natural injection 
\begin{equation} \label{affine-to-Hodge}
H^{q}(B,\iota_*\bigwedge^p\Lambda)\ra H^q(X,\Omega^p)
\end{equation}
with precise understanding of the cokernel. A criterion for \eqref{affine-to-Hodge} to be bijective was given in \cite[Theorem 3.21]{logmirror2}.
The log Hodge groups are relevant because they are isomorphic to the actual Hodge groups whenever $X$ appears in a family with smooth or orbifold nearby fibres, e.g. in the family \eqref{GS-family}, see \cite{logmirror2,smoothtor}. 
For our introductory K3 example, the relationship of $H_{1,1}$ to the Picard group was discussed in \cite[\S1.4.3]{RS20}, see also \cite{logmirror1}. 

There is a related notion of \emph{tropical (co)homology} due to \cite{IKMZ} for which a correspondence result like \eqref{affine-to-Hodge} has also been proved. 
The context here is that, instead of having $B$ as a topological manifold, \cite{IKMZ} replace $B$ by a tropical variety $V$ in $\RR^m$. 
A tropical variety $V$ is naturally an integral affine manifold on a dense open set of $V$, namely the union of the interiors of all maximal polyhedra that $V$ consists of. Hence, the sheaf $\Lambda$ is defined on this open set as before. 
However, the singularities of $V$ are a lot more severe compared to Definition~\ref{def-affmfd}.
The correct replacement of $\iota_*\bigwedge^p\Lambda$ in the situation of a tropical variety $V$ is denoted $^\ZZ\shF_p$ in \cite{IKMZ}. 
The (co)homology groups of $^\ZZ\shF_p$ are called tropical (co)homology.
The Picard group as well as a pairing similar to \eqref{main-pairing} has been understood in this context in \cite[Definition 5.2 and Theorem 5.3]{JRS}. 
In the special situation where $V$ a tropical Calabi--Yau hypersurface, it is possible to collapse its unbounded parts and obtain an integral affine manifold with singularities in our sense, see \cite[\S3]{Ya}. The relation between the two notions of tropical versus affine homology is not yet properly understood.

Beyond Calabi--Yau tropical varieties, there should yet be another relationship between tropical homology and affine homology:
given a general tropical hypersurface $V$ in $\RR^m$, we may place an integral affine structure on $B:=\RR^m\times\RR$ that makes $V\times\{0\}$ be the discriminant $\Delta$ by a construction similar to Construction~\ref{construction-symple-model}.
We expect that the resulting $(B,\Delta)$ satisfies the aforementioned criterion in \cite[Theorem 3.21]{logmirror2} if and only if $V$ satisfies the smoothness notion in \cite{IKMZ}.

A sheaf-theoretic approach to \cite{IKMZ} was recently given in \cite{GSh} and refined Hodge theoretic properties proved in \cite{AP20}.

\subsection{Tropical intersection theory}
Another relevant application of the pairing \eqref{main-pairing} is by means of \emph{tropical intersection theory}, a subject that has already been studied for tropical subvarieties in $\RR^m$ in \cite{AR10} and for tori in \cite[\S6.2]{eigenwave}
 but not for tropical cycles in a general integral affine manifold with singularities $B$.
Given two tropical subvarieties $V,W\subset \RR^m$ of complementary dimension in general position, it is straightforward to define their intersection number $V.W\in\ZZ$ because this amounts to locally intersecting integrally spanned oriented linear spaces of complementary dimension. It is much less obvious that the resulting number is independent of translating  $V$ or $W$, let alone deformation. This problem was solved in \cite[Proposition 9.11]{AR10} when $B=\RR^m$. We prove the well-definedness of the intersection number for more general $B$, a result that has already been used in \cite{NOR,tropLag}.

\begin{definition} 
\label{def-int-prod}
Let $B$ be a compact integral affine $n$-manifold with singularities and $\Omega$ a primitive global section of $\iota_*\bigwedge^n\Lambda$, i.e.~$B$ is oriented. 
Set $B^\circ = B\setminus\partial B$ and $H^\partial_{p,q}:=H_q(B,\partial B;\iota_*\bigwedge^{p}\Lambda)$.
Let $$D\colon H_{p,q}\ra H^{n-q}_c(B^\circ,\iota_*\bigwedge^{p}\Lambda)\qquad\hbox{ and }\qquad
D^\partial\colon H^\partial_{n-p,n-q}\ra H^{q}(B,\iota_*\bigwedge^{n-p}\Lambda)$$ 
denote the Poincar\'e--Lefschetz isomorphisms of Theorem~\ref{thm-lefschetz}. 
For $\alpha\in H_{p,q}$ and $\beta\in H^\partial_{n-p,n-q}$ we define\footnote{Some authors invert the order of $\alpha,\beta$, see e.g.~\cite[VI-Example~11.12]{Br93}.} the bilinear intersection product $\alpha\cdot\beta\in\ZZ$ by
$$ \alpha\cdot\beta := D(\alpha)\cup D^\partial(\beta)\in H^{n}_c(B^\circ,\iota_*\bigwedge^{n}\Lambda)\underset\Omega\cong\ZZ.$$
\end{definition}

For two oriented subspaces $V,W$ in an oriented vector space $U$ with $V\oplus W=U$, we define the sign $\eps(V,W)\in\{-1,1\}$ by 
$$\orr_V\wedge \orr_W=\eps(V,W)\orr_U$$ where $\orr_V,\orr_W,\orr_U$ denote the orientations of $V,W,U$ respectively. More generally, we set $\eps(V,W)=0$ if $V\oplus W\ra U$ is not an isomorphism.
\begin{theorem} 
\label{thm-intersect}
In the situation Definition~\ref{def-int-prod}, assume we are given cycles $V\in H_{p,q}$, $W\in H^\partial_{n-p,n-q}$ that meet transversely in a finite set of points $V\cap W$ which is disjoint from $\Delta$ and $\partial B$. 
Assume further that each point lies in the interior of a maximal cell of $V,W$ respectively with a well-defined tangent space at the point, then the following integers coincide
\begin{enumerate}
\item The intersection product $V\cdot W$,
\item The image of $V\otimes D^\partial(W)$ under the pairing \eqref{eq-gen-pairing} after inserting the isomorphism $\iota_*\bigwedge^{n-p}\Lambda\ra \iota_*\bigwedge^{p}\check\Lambda$ given in Lemma~\ref{lemma-sheaf-pairing-perfect}.
\item The intersection number
$$\sum_{x\in V\cap W} \eps(T_{V,x},T_{W,x}) \frac{\xi^V_x\wedge \xi^W_x}\Omega$$
where $\xi^V_x\in \bigwedge^p\Lambda_x$,~$\xi^W_x\in \bigwedge^{n-p}\Lambda_x$ denotes the coefficient of $V,W$ at $x$ respectively and $\frac{\xi^V_x\wedge \xi^W_x}\Omega$ is the integer $k$ so that $\xi^V_x\wedge \xi^W_x=k\Omega_x$.
\end{enumerate}
\end{theorem}

\begin{theorem}\label{thm-zharkov}
Assume $B$ is an integral affine manifold together with a polyhedral decomposition and multivalued strictly convex function. 
If $B$ is symple then \cite{RZ} gives a topological $2n$-orbifold $X$ with surjection to $B$ that compactifies the $n$-torus-bundle $T^*_{B\setminus\Delta}/\check\Lambda\ra B\setminus\Delta$. 
Let $V\in H_{p,q}$, $W\in H_{n-p,n-q}$ be as in Theorem~\ref{thm-intersect} and let $\beta_V\in H^{n-p+q}(X,\ZZ)$, $\beta_W\in H^{n+p-q}(X,\ZZ)$ be lifts of $r_{p,q}(V)$ and $r_{n-p,n-q}(W)$ respectively.
We have
$$V\cdot W=(-1)^{(n-p)(q-1)}\beta_V.\beta_W,$$
i.e. the intersection number of Definition~\ref{def-int-prod} agrees up the factor $(-1)^{(n-p)(q-1)}$ with the ordinary topological intersection number (which is the fundamental class coefficient of the cup product $\beta_V\cup \beta_W$).
\end{theorem}

\begin{remark} 
In the situation of Theorem~\ref{thm-zharkov}, the Hodge groups on the right hand side of \eqref{affine-to-Hodge} are expected to inject into the orbifold de Rham cohomology $H^{p+q}_{orb}(X)$ of $X$, so in particular the affine Hodge groups inject into $H^{p+q}_{orb}(X)$.
One may speculate that Poincar\'e duality for $X$ combined with Theorem~\ref{thm-zharkov} leads to a proof of the perfectness of \eqref{eq-gen-pairing} (and another proof of Theorem~\ref{main-thm-intro}) under the hypothesis of Theorem~\ref{thm-zharkov}. In any event, the map into $H^{p+q}_{orb}(X)$ gives a new geometric interpretation for the pairing \eqref{main-pairing}.
\end{remark}

\subsection{Focus-focus points and goggles}
To see Theorem~\ref{thm-intersect} in action, we return to the motivational example of the K3 fibration. 
The affine structure of a neighborhood of the discriminant point of $\Delta$ at an $I_1$-fiber of an elliptic fibration features a \emph{focus-focus} singularity (\cite{Wi},\cite{D},\cite[Definition\,4.2]{Sy01},\cite[Example 1.16]{logmirror1}), see Figure~\ref{fig-focusfocus}. Parallel transport along a simple clockwise loop around the singularity has linear part $A=\bigl( \begin{smallmatrix}1 & 1\\ 0 & 1\end{smallmatrix}\bigr)$ in a suitable oriented basis $e_1,e_2$ (which is recognizable as the Dehn twist on cohomology). 
\begin{figure}[h]
%\captionsetup{width=1.4\linewidth}
%\begin{center}
\includegraphics[width=0.50\linewidth,bb=0 0 80 35]{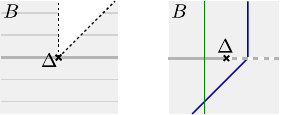}
\captionsetup{width=.95\linewidth}
\caption{Two differently slit neighborhoods of a focus-focus singularity $\Delta\in B$ with their embedding in $\RR^2$. The two dashed lines on the left get identified in $B$. The gray horizontal lines on the left are straight lines in the affine structure, as are the blue and green lines on the right. The bold gray line is the monodromy invariant direction at the singularity respectively. 
The fact that parallel lines intersect after passing the invariant line of the singularity is reminiscent of the effect of a lens which the author likes to think justifies the name \emph{focus-focus}.}
\label{fig-focusfocus}
\end{figure}

The first homology and cohomology of $\iota_*\Lambda$ and $\iota_*\check\Lambda$ on a neighborhood of the focus-focus-singularity is trivial, so in order to have a non-trivial demonstration of the pairing \eqref{main-pairing}, we need an affine manifold $B$ with two focus-focus points. 
The rank of the first homology/cohomology is then again trivial unless the two monodromy invariant directions are parallel to one another. 
Figure~\ref{fig-brille} on the left shows an example of two focus-focus points with the same invariant line and on the right with two distinct parallel lines. For the sake of intuition building, we now prove Theorem~\ref{main-thm-intro} for the examples in Figure~\ref{fig-brille} by direct computation.
\begin{figure}
\captionsetup{width=0.96\linewidth}
\begin{center}
\includegraphics[width=1.0\linewidth]{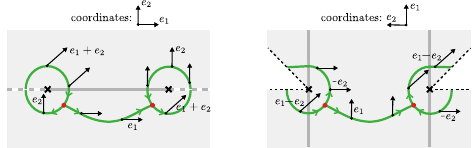}
\end{center}
\captionsetup{width=.95\linewidth}
\caption{Two different examples of an integral affine manifold $B$ with two focus-focus singularities (black crosses) showing the slit up $B$ embedded in $\RR^2$ respectively. 
Invariant lines are bold gray, slits are dashed. 
The green cycles are generators of $H_1(B,\iota_*\Lambda)$ respectively. 
Black arrows indicate the sections of $\Lambda$ attached to the the green oriented edges of the cycles.
At each of the (red) vertices of the green cycles, one checks that the oriented sum of the adjacent black arrows is zero as necessary for a cycle. We call these types of cycles \emph{easy cycles} or \emph{goggles}.}
\label{fig-brille}
\end{figure}
The computation will also verify that the green cycles in the figure are generators of $H_1(B,\iota_*\Lambda)$ respectively.
We first compute $H^1(B,\iota_*\check\Lambda)$ by a \v{C}ech cover that features two open sets, each containing one of the focus-focus points.
Let $e_1,e_2$ be the basis for a stalk of $\Lambda$ as indicated in the figure and let $e^*_1,e^*_2$ be the dual basis generating a stalk of $\check\Lambda$. 
The sections of $\iota_*\Lambda$ on the open sets are the invariant lattice under $A^T=\bigl( \begin{smallmatrix}1 & 0\\ 1 & 1\end{smallmatrix}\bigr)$ respectively, so this is $\ZZ e_2^*$. The sections on the intersection of the two open sets are however all of $\check\Lambda$ in both examples, so we conclude
$$ H^0(B,\iota_*\check\Lambda)=\ZZ e_2^*,\qquad H^1(B,\iota_*\check\Lambda)=\ZZ e_1^*.$$
A very similar computation also gives
\begin{equation}\label{eq-example-Lambda}
 H^0(B,\iota_*\Lambda)=\ZZ e_1,\qquad H^1(B,\iota_*\Lambda)=\ZZ e_2.
\end{equation}
Generally speaking, singular homology with sheaf coefficients is hard to compute directly. We provide simplicial techniques for this in Appendix~\ref{general-constructible-sheaves} but for now, we use 
that $B$ is orientable, so we may apply Theorem~\ref{thm-lefschetz} and then a \v{C}ech cohomology computation gives the homology. 
By the same \v{C}ech cover that we just used to compute \eqref{eq-example-Lambda} (though leading to a different complex), we compute that $H^0(\partial B,\Lambda|_{\partial B})=\ZZ e_1$ and $H^1(\partial B,\Lambda|_{\partial B})=\ZZ e_2$ and most importantly that the restriction $H^1(B,\iota_*\Lambda)\ra H^1(\partial B,\Lambda|_{\partial B})$ is the zero map! 
The restriction $H^0(B,\iota_*\Lambda)\ra H^0(\partial B,\Lambda|_{\partial B})$ is an isomorphism, so the long exact sequence of the pair $(B,\partial B)$ yields $H^1(B,\partial B;\iota_*\Lambda)=\ZZ e_2$ and thus by Theorem~\ref{thm-lefschetz} we find $H_1(B,\iota_*\Lambda)\cong\ZZ$. We show that the green cycle in Figure~\ref{fig-brille} is a generator by pairing it with
the generator $e_1^*$ of $H^1(B,\iota_*\check\Lambda)$ under \eqref{main-pairing}. 
This works by evaluating $e_1^*$ on the section of the center edge (up to taking the orientation of the edge into consideration by a sign) because this edge is the only part of the green cycle meeting the intersection of the two open charts. Since the section on the center edge is $e_1$ in both examples, we conclude that the pairing \eqref{eq-bil-form} is perfect even over $\ZZ$ in the examples of Figure~\ref{fig-brille}.

\begin{remark} 
Positioning two focus-focus singularities with non-parallel invariant directions but so that their primitive direction generators don't span $\Lambda_b$ (for $b$ a general point) yields an example of $H^1(B,\iota_*\check\Lambda)$ that is torsion as can easily be computed by a \v{C}ech cover with two charts as just done. 
Furthermore, again by a similar computation as just done, one finds that $H_1(B,\iota_*\Lambda)$ is also torsion. Of course, there is no non-trivial map $T'\otimes T\ra \ZZ$ for $T,T'$ torsion groups, so it makes sense to study the perfectness question of \eqref{main-pairing} over $\QQ$ rather than $\ZZ$.
\end{remark}

\begin{figure}[h]
%\captionsetup{width=1.4\linewidth}
\includegraphics[width=.78\linewidth]{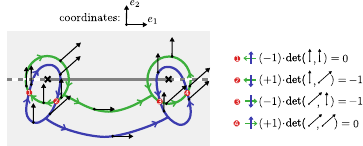}
\caption{An example computation using Theorem~\ref{thm-intersect}\,(3): For $\beta$ a goggle (introduced below in Figure~\ref{fig-brille}) and $\beta'$ a small perturbation of $\beta$, the above illustration shows the computation of $\beta.\beta'=-2$.}
\label{fig-brille-self-cut}
\end{figure}

\subsection{The K3 lattice made of goggles}
An elliptic K3 fibration over $\PP^1$ with at worst $I_1$ fibers needs to have precisely 24 singular fibers to match the Euler characteristic.
Similarly, if the only type of singularity permitted for an integral affine structure on a 2-sphere is the focus-focus type then there will be 24 singular points. This latter statement however requires a different proof using facts on central extensions of $\operatorname{PSL}_2(\ZZ)$ making the result a Gauss--Bonnet type theorem, see \cite{LM} or \cite[6.5, Theorem 2]{KS}, \cite[Theorem 6.7]{LS}, \cite[Corollary 1.19]{Th}.

There are numerous ways to obtain specific integral affine structures with 24 singularities on $S^2$ from toric degenerations of K3 surfaces.
For an example, consider a smooth tri-degree $(2,2,2)$-hypersurface $X$ in $(\PP^1)^3$. 
Degenerate the hypersurface in a pencil to the toric boundary divisor $D$, that is $tX+D$ for $t$ a parameter that gives the degeneration at $t=0$.
The six copies of $\PP^1\times\PP^1$ that are the components of $D$ form a ``cube''. 
The moment map $\PP^1\ra[0,1]$ is $[z:w]\mapsto |z|^2/(|z|^2+|w|^2)$. 
By taking products, we get the moment map $\mu\colon (\PP^1)^3\ra[0,1]^3$ which sends $D$ onto $\partial([0,1]^3)$, the boundary of an actual cube, hence to a topological $S^2$.
The composition 
\begin{equation}
f\colon X\ra D\ra (\hbox{boundary of a 3-cube})=:B \label{cube-SYZ-map}
\end{equation}
of a retraction $X\ra D$ with the moment map is a 2-torus fibration with 24 $I_1$ fibers (as long as $X$ is sufficiently general).
The intersection of $X$ with $\Sing(D)$ in $(\PP^1)^3$ is a set of 24 points. Let $\Delta$ denote its image in $B$ under $\mu$ giving two points for each edge in $B$.
The map $f$ in \eqref{cube-SYZ-map} can be upgraded to a symplectic fibration, so we obtain an integral affine structure on $B\setminus \Delta$.
This integral affine structure can be covered by three charts as is shown in 
Figure~\ref{fig-cubecharts}.
Each singularity is a focus-focus point with invariant monodromy direction given by the edge it lies on.
\begin{figure}
%\captionsetup{width=1.4\linewidth}
%\begin{center}
%\centering
$\vcenter{\hbox{
\includegraphics[width=0.27\linewidth]{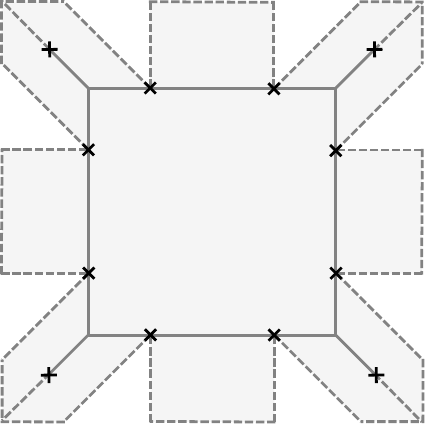}
}}$
$\vcenter{\hbox{
\includegraphics[width=0.68\linewidth]{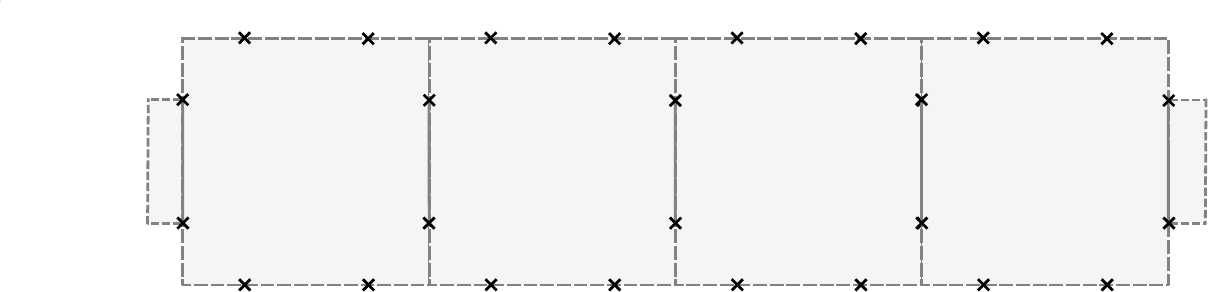}
}}$
%\end{center}
\captionsetup{width=.95\linewidth}
\caption{Gluing a ``cube'': two copies of the left chart (for top and bottom face of a cube) together with the right chart (for the equator) cover the complement of 24 points (black crosses) on $S^2$ and induce an integral affine structure with 24 focus-focus singularities on $S^2$.}
\label{fig-cubecharts}
\end{figure}
Our next goal for our cube example is to exhibit a lattice basis of $H_1(B,\iota_*\Lambda)$ in terms of goggles. Symington \cite[Figure 17]{Sy01} gave a similar basis for a different example. By \cite[Theorem on p.225]{LM} all monodromy representations of affine structures on $S^2\setminus \{24\hbox{ focus-focus points}\}$ are conjugate, so a basis like Symington's must exist also in our example. Figure~\ref{goggleE8} shows eight goggles that form a $-E_8$. 
To actually verify this, we need to compute the self-intersection of a goggle which is done in Figure~\ref{fig-brille-self-cut} below in the context
\noindent
\begin{minipage}[b]{0.53\textwidth}\vspace{1.6mm}
motivating the intersection theorem.
That physical intersection points of two different goggles result in an intersection number $1$ is computed similarly.
Since there are two copies of the chart that we used to obtain the $-E_8$, we actually obtain $(-E_8)^{\oplus 2}\subset H_1(B,\iota_*\Lambda)$.
It remains to exhibit two copies of the hyperbolic plane $H$ by means of new goggles that don't meet existing ones.
This is demonstrated in Figure~\ref{fig-equator-wcycles}. The two goggles shown here ``use'' the focus-focus\hspace*{\fill} points\hspace*{\fill} that\hspace*{\fill} had\hspace*{\fill} been\hspace*{\fill} spared\hspace*{\fill} from\hspace*{\fill} the 
$-E_8$-constructions. Furthermore there are two cycles that don't encircle any single focus-focus point but traverse along the entire equator. These have self-intersection zero since they can easily be displaced.
One checks that
\end{minipage}
\qquad
%second column
\begin{minipage}[b]{0.4\textwidth}
%\vspace{-2ex}
\includegraphics[width=1\textwidth]{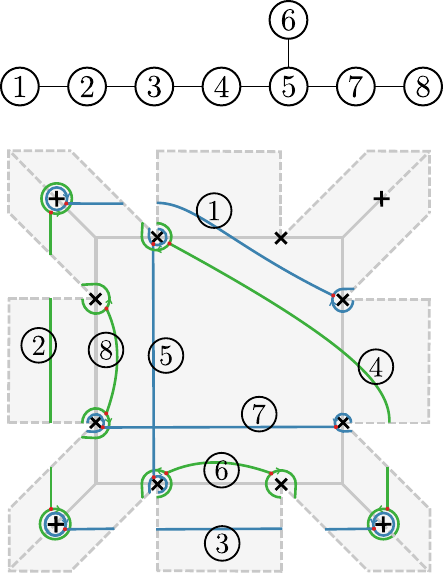}
\captionsetup{width=1.4\linewidth}
\captionof{figure}{$-E_8$ made of goggles} \label{goggleE8}
\ 
\end{minipage}\\
\noindent 
a pair of goggle and equator each form the intersection pairing $\bigl( \begin{smallmatrix}0 & 1\\ 1 & -2\end{smallmatrix}\bigr)$ which is equivalent to $H$ by a change of basis.
We have constructed $(-E_8)^{\oplus 2}\oplus H^{\oplus 2}\subset H_1(B,\iota_*\Lambda)$ and the key question is why this inclusion would be an isomorphism. 
\begin{figure}[h]
%\captionsetup{width=1.4\linewidth}
\includegraphics[width=1\linewidth]{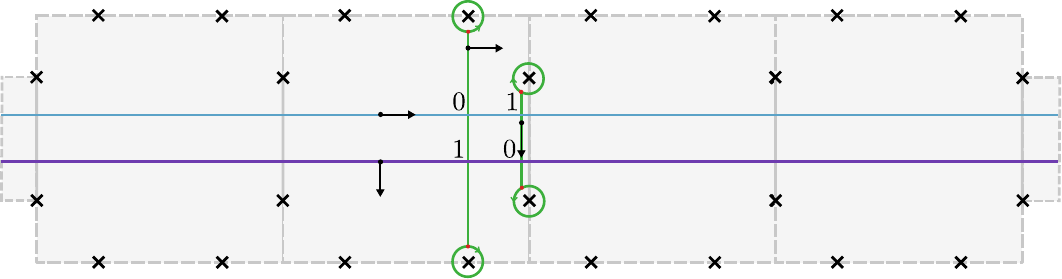}
\captionsetup{width=.96\linewidth}
\caption{Two hyperbolic planes, each made of a goggle and an equator. Numbers at intersection points of cycles display the respective the intersection numbers.}
\label{fig-equator-wcycles}
\end{figure}
An easy way to argue this is to follow the arguments at the beginning of the article to know that $H_1(B,\iota_*\Lambda)$ itself is isomorphic to $(-E_8)^{\oplus 2}\oplus H^{\oplus 2}$.
Symington's basis in \cite[Figure 17]{Sy01} is given in terms of relative cycles, that is cycles in $H_1(B,\Delta;\iota_*\Lambda)$. Note that the stalk of $\iota_*\Lambda$ at a point in $\Delta$ has rank one given by the monodromy invariant vectors. An edge of a cycle that ends on such a point therefore necessarily needs to carry a monodromy invariant vector.

There is a split exact sequence
\begin{equation}
0\ra H_1(B,\iota_*\Lambda)\ra H_1(B,\Delta;\iota_*\Lambda) \ra \underbrace{H_0(\Delta;\iota_*\Lambda)}_{\cong \ZZ^{24}}\ra 0.
\label{towardsrelatives}
\end{equation}
Indeed, Figure~\ref{fig-goggle-to-relative} sketches how we can lift the relative cycle shown on the right to a goggle by subtracting a ``local relative cycle'' for each point in $\Delta$. 
Such local relative cycles split \eqref{towardsrelatives}. 
\begin{figure}[h]
%\captionsetup{width=1.4\linewidth}
\includegraphics[width=1\linewidth]{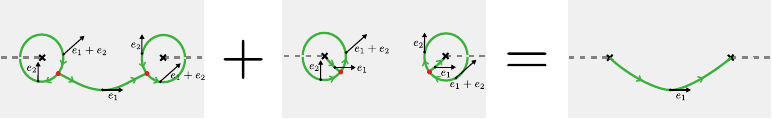}
\captionsetup{width=.96\linewidth}
\caption{Goggle plus two local relative cycles yield a Symington relative cycle.}
\label{fig-goggle-to-relative}
\end{figure}
The group $H_1(B,\Delta;\iota_*\Lambda)$ plays a central role in \cite{tropK3}.

\subsection*{Acknowledgments} 
The interest in expressing the pairing \eqref{eq-bil-form} purely in terms of affine geometry arose from discussions with Diego Matessi who already had an idea of its local form in dimension two.
Bernd Siebert greatly encouraged this work, provided background knowledge on the K3 case, technical advice on sheaf homology and gave valuable suggestions for improving the presentation.
Andreas Gross, Johannes Rau, Kristin Shaw and Yuto Yamamoto guided me through the state of knowledge of tropical intersection theory. 
The main theorem was proved during my stay as a postdoc at the Fields institute in Toronto in 2013, my gratitude here goes to Noriko Yui. 
I like to thank the referee for many improvements and corrections.

%===========================================================
%
%		section: Symple models and symple affine manifolds with singularities
%
%===========================================================

\section{Symple models and symple affine manifolds with singularities} \label{sec-simple}
Our definition of simplicity is derived from \cite[Definition\,1.60]{logmirror1} while being more general and easier to describe. 
We fix $a,b>0$.
Let $\triang$ be an $a$-dimensional lattice simplex\footnote{A lattice polytope is the convex hull in $\RR^n$ of a finite subset of $\ZZ^n$. Two such polytopes are considered isomorphic if a transformation in the affine group $\GL_n(\ZZ)\ltimes\ZZ^n$ takes one bijectively to the other.} in $\RR^a$ and let $\check\triang$ be a $b$-dimensional lattice simplex in $(\RR^b)^*=\Hom(\RR^b,\RR)$, both with respect to the standard lattice.
Let $Y\subset (\RR^a)^*$ and $\check Y\subset \RR^b$ be the codimension one skeleton of the normal fan of $\triang, \check\triang$ respectively.
\begin{definition} 
\label{def-simple-model}
Given $\triang$, $\check\triang$ as above\footnote{The author apologizes for using $\Delta$ and $\triang$ to mean different objects in the following. The benefit is to be consistent with the notation in \cite{logmirror1}.}, a {\bf symple model} is an affine manifold with singularities $(B,\Delta)$ homeomorphic to $\big((\RR^a)^*\times\RR^b,Y\times\check Y\big)$
so that the linear part of its\hfill monodromy\hfill representation\hfill is\hfill given\hfill as\hfill follows.\hfill If\hfill $(x,y)\in Y\times\check Y$\hfill is\hfill a\hfill point\hfill in\hfill a
\\[5pt]
\noindent
\begin{minipage}{\linewidth} 
\begin{minipage}{0.65\textwidth} 
maximal stratum $C$, the monodromy along a simple loop around $C$ is given as an automorphism of $(\RR^a)^*\times \RR^b$ by
\begin{equation}
\label{eq-explicit-monodromy}
 v \mapsto v+ \langle v,n_y\rangle m_x
\end{equation}
where $n_y$ is the (integral) edge vector of $\check\triang$ for the edge that is dual to the stratum of $\check Y$ containing $y$ and 
$m_x$ is the edge vector of $\triang$ that is dual to the stratum of $Y$ containing $x$ (and we identify $n_y$ with $(0,n_y)$ and $m_x$ with $(m_x,0)$). 
Orientations of the vectors and the loop match the illustration on the right after taking a suitable 2-dimensional slice transverse to $C$. 
\end{minipage}
\hspace{0.05\textwidth}
\begin{minipage}{0.3\textwidth}
\begin{figure}[H]
\includegraphics[width=1\textwidth]{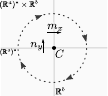}
\end{figure}
\vspace{2pt} 
\end{minipage}
\end{minipage} \vspace{0pt}
\end{definition}
If $\dim B=2$, then $\dim\Delta=0$ and the monodromy around the point $\Delta$ in a symple model for a suitable basis of a stalk of $\Lambda$ is given by $\bigl( \begin{smallmatrix}1 & k\\ 0 & 1\end{smallmatrix}\bigr)$ where $k$ is the product of the integral lengths of the two intervals $\triang$ and $\check\triang$.

If $\dim B=3$ and $B$ is not already given as a product of the two-dimensional situation with $\RR$, then either $\triang$ is a two-simplex or $\check\triang$ is a two-simplex and the respective other one is an interval, see Figure~\ref{fig-explicitmodels}. In both cases, $\Delta$ is homeomorphic to the 1-skeleton of the fan of $\PP^2$, i.e. three different rays that originate in the same point. At a point along a ray, the model is locally isomorphic to the product of $\RR$ with the situation for $\dim B=2$ but there is an interesting interplay between the monodromies around the three rays that is encoded in the polytopes $\triang,\check \triang$. 
The special case for $\dim B=3$ known by the name ``simple'' is given when $\triang$ and $\check\triang$ are both standard simplices\footnote{\label{fnstandard}The standard simplex of dimension $k$ is the convex hull of $\{0,e_1,...,e_k\}$ in $\RR^k$.}. In particular, $k=1$ for each ray in the simple case and this situation has been studied extensively in \cite[Theorem 1.10]{Gr98} and \cite[Example 2.8, Example 2.10]{CBM09}.

\begin{construction} 
\label{construction-symple-model}
A symple model exists for every pair $\triang,\check\triang$ as follows.
Pick a homeomorphism $h:(\RR^a)^*\ra \op{int}(\triang)$ which identifies $Y$ with the union of those simplices in the first barycentric subdivision of $\triang$ that do not meet any vertex of $\triang$. This can be done so that the closure in $\triang$ of $h(Y_e)$ meets $e$ whenever $Y_e$ is the maximal stratum of $Y$ that corresponds to an edge $e$.
Take $B=\op{int}(\triang)\times\RR^b$ and $\Delta=h(Y)\times\check Y$. 
Given a vertex $v\in\triang$, we define the open set
$$ U_v:=B\setminus \big(\Delta+\RR_{\ge0}(\triang-v)\big)$$
and observe that each $U_v$ is contractible and $B\setminus\Delta=\bigcup_v U_v$, see Figure~\ref{fig-explicitmodels}.
\begin{figure}
\includegraphics[width=.8\textwidth]{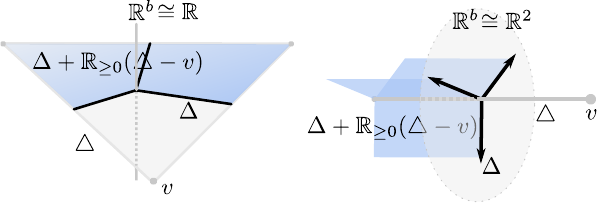}
\captionsetup{width=.95\linewidth}
\caption{In the explicit symple model for $\dim B=3$: depicting the set $\Delta+\RR_{\ge0}(\triang-v)$ when $\dim\triang=2$ (left) and $\dim\triang=1$ (right).}
\label{fig-explicitmodels}
\end{figure}
We fix one vertex $v_0\in\triang$. 
Let $\psi_{\check\triang}:\RR^b\ra\RR$ be the piecewise linear function defined by $\check\triang$ via
$$\psi_{\check\triang}(y)=-\min\{\langle y,n\rangle\,\mid\, n\in \check\triang\}.$$
The function is linear on each connected component of $\RR^b\setminus\check Y$ and in fact given by the negative of the pairing with the vertex of $\check\triang$ that corresponds to that component.

The set $U_v$ already has an affine structure from its embedding in $\op{int}(\triang)\times\RR^b$ but that is not the one we want to use unless $v=v_0$.
To obtain an affine structure on $U_v$, we define the chart $\varphi_v:U_v \ra \RR^a\times\RR^b$ by means of the piecewise linear embedding given by
$$(x,y)\mapsto (x,y)+\psi_{\check\triang}(y) \cdot (v-v_0,0).$$

The transition maps $\varphi_v(U_v\cap U_w) \stackrel{\varphi_{vw}}\ra \varphi_w(U_v\cap U_w)$ between charts are now forced upon us to be
$\varphi_{vw}\colon (x,y)\mapsto (x,y)+\psi_{\check\triang}(y) \cdot (w-v,0)$.
For $v\neq w$ the intersection $U_v\cap U_w$ agrees with $B\setminus(\op{int}(\triang)\times \check Y)$ and therefore
$\psi_{\check\triang}$ is linear on each component of $U_v\cap U_w$ as needed.
Figure~\ref{fig-explicitmodelcharts} illustrates the charts for the situation in dimension two.
We leave it to the reader to verify that the monodromy around each maximal stratum of $\Delta$ agrees with \eqref{eq-explicit-monodromy} (which can be done by reduction to the two-dimensional case by means of a suitable transverse slice).
\begin{figure}
\includegraphics[width=.9\textwidth]{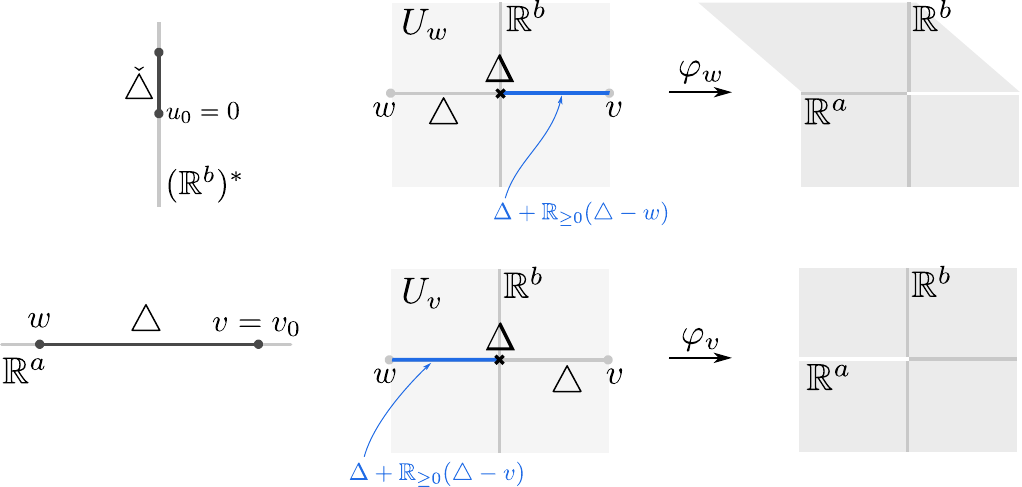}
\captionsetup{width=.95\linewidth}
\caption{For $\dim B=2$ and $a=b=1$, depicting $\varphi_v,\varphi_w$ when $v=v_0$.}
\label{fig-explicitmodelcharts}
\end{figure}
\end{construction}

For the remainder of the section, we will use $y$ to reference points in $B$.
\begin{definition}
 \label{definition-simple}
An integral affine manifold with singularities and possibly boundary $(B,\Delta)$ is {\bf symple} if it is locally a product of symple models and a trivial factor.
Specifically, this means that for every point $y\in B$ there exist $a_1,b_1,...,a_r,b_r,c,c'\ge 0$ and lattice polytopes $\triang_j\subset \RR^a$, $\check\triang_j\subset (\RR^b)^*$ for $j=1,...,r$ so that, in a neighborhood of $y$, $(B,\Delta)$ is homeomorphic to 
$((\RR^{a_1})^*\times \RR^{b_1})\times...\times ((\RR^{a_r})^*\times \RR^{b_r})\times \RR^c\times \RR_{\ge 0}^{c'}$ identifying $\Delta$ with the union of the inverse images of $Y_j\times\check Y_j$ under the projections to 
$((\RR^{a_j})^*\times \RR^{b_j})$ for $j=1,...,r$. The homeomorphism identifies the local systems $\Lambda$ on the complements of the singularities.
\end{definition} 
Without loss of generality, $c'\in\{0,1\}$ because $\RR_{\ge 0}^2$ is homeomorphic to $\RR\times\RR_{\ge 0}$ etc. and the monodromy doesn't interact with the boundary of $B$.

A large host of simple examples (and thus symple examples) in all dimensions can be found in \cite[Theorem 3.16]{grossannalen}. If the subdivision that underlies loc.cit.~is only a triangulation as opposed to a maximal triangulation, the resulting affine manifold will be symple but not simple.

\begin{remark} 
\label{wlog-p-origin}
If $(B,\Delta)$ is symple then every point $y$ permits a symple model that identifies $y$ with the origin by a ``localization argument'': if $y$ is a point in a symple model given by $\triang,\check\triang$ so that $y$ is \emph{not} the origin, then a suitable open neighborhood of $y$ is the product of $\RR^c$ with a symple model for suitable faces $\triang'\subset \triang, \check\triang'\subset \check\triang$ and $c$ is the sum of the co-dimensions of $\triang'$ and $\check\triang'$. The faces $\triang',\check\triang'$ are precisely the ones that correspond to the maximal stratum of $Y$ and $\check Y$ that contain $y$. (If $y$ doesn't lie in $\Delta$ then $c=\dim B$.)
\end{remark}

For $(B,\Delta)$ symple, $\dim B=n$, the linear part of the monodromy is contained in $\GL_n(\ZZ)$. 
As in Definition~\ref{def-affmfd}, let $\Lambda$ and $\check\Lambda=\Hom(\Lambda,\ZZ)$ denote local systems of integral tangent vectors and their duals on $B\setminus\Delta$ and let $\iota\colon B\setminus\Delta\ra B$ denote the inclusion.
\begin{remark} 
\label{rem-dir-sum}
For $(B,\Delta)$ symple and $y\in B$, the restriction of $\iota_*\Lambda$ to a neighborhood of $y$ decomposes as a direct sum
$\iota_*\Lambda=\Lambda_1\oplus...\oplus\Lambda_r\oplus\ZZ^c$ where $\Lambda_j$ is the pullback of $(\iota_j)_*\Lambda'_j$ under the projection to $B_j:=((\RR^{a_j})^*\times \RR^{b_j})$ for $\iota_j\colon B_j\setminus (Y_j\times\check Y_j)\hra B_j$ the inclusion and $\Lambda'_j$ integral tangent vectors on $B_j\setminus (Y_j\times\check Y_j)$. 
\end{remark}

The following Lemma makes the key ingredient for the construction of the pairing \eqref{eq-gen-pairing}. 
As can be seen from its proof, the Lemma holds a lot more generally than for the assumptions stated but we stick to our setup.

\begin{lemma} 
\label{lemma-sections} 
Let $(U,\Delta)$ be an integral affine manifold $U$ with singularities $\Delta$ and assume $\dim U>1$. Let $\pi_1=\pi_{1}(U\setminus\Delta,y)$ denote the monodromy group for some base point $y\in U\setminus\Delta$, then
$$H^0(U,\iota_*\bigwedge^p\check\Lambda)=\Big(\bigwedge^p\check\Lambda_y\Big)^{\pi_1}.$$
If $y\in \partial U\setminus \Delta$ and $\partial U\setminus \Delta$ contains a set of generators for $\pi_1$, e.g. when U is the natural closed ball compactification of a symple model, then also
$$H^0(\partial U,\iota_*\bigwedge^p\check\Lambda) = \Big(\bigwedge^p\check\Lambda_y\Big)^{\pi_1}.$$
Similar statements holds with $\check\Lambda$ replaced by $\Lambda$. 
Furthermore, there is a natural commutative diagram
\begin{equation}
\label{eq-key-map}
\begin{gathered}
\xymatrix@C=30pt
{ 
H_0(U,\iota_*\bigwedge^p\Lambda)\ar[r]\ar[d] & \Hom(H^0(U,\iota_*\bigwedge^p\check\Lambda),\ZZ)\ar[d]\\
H_0(U,\iota_*\bigwedge^p\Lambda_\QQ)\ar[r]  & \Hom(H^0(U,\iota_*\bigwedge^p\check\Lambda_\QQ),\QQ)
}
\end{gathered}
\end{equation}
with the bottom horizontal map being an isomorphism. 
A similar diagram exists if we swap $\Lambda$ and $\check\Lambda$.
\end{lemma}
\begin{proof} 
The statement $H^0(U,\iota_*\bigwedge^p\check\Lambda)=H^0(U\setminus\Delta,\bigwedge^p\check\Lambda)$ is clear. 
Note that $\bigwedge^p\check\Lambda$ is a local system on the path-connected space $U\setminus\Delta$.
Combining the equivalence between local systems and representations of the fundamental group, e.g.~\cite[Lemma 4.7]{DK02}, with \cite[Proposition 5.14,\,2.]{DK02} gives 
$H^0(U\setminus\Delta,\bigwedge^p\check\Lambda)=\Big(\bigwedge^p\check\Lambda_y\Big)^{\pi_1}$. 
The second statement where $\partial U$ replaces $U$ follows by the same arguments.

Let $\gamma_1,...,\gamma_r$ be a set of generators for $\pi_1$ and let $T_i$ be the monodromy transformation of $\bigwedge^p\Lambda_y$ given by parallel transport along $\gamma_i$.
Let $T^t_i$ denote the transformation of $\bigwedge^p\check\Lambda_y$ that is dual to $T_i$. We have exact sequences
$$
\xymatrix@C=30pt
{ 
\bigwedge^p\Lambda_y\oplus...\oplus \bigwedge^p\Lambda_y \ar^(.6){\sum_i(T_i-\id)}[rr] && \bigwedge^p\Lambda_y \ar[r] & H_0(U,\iota_*\bigwedge^p\Lambda) \ar[r] & 0,
}
$$
$$
\xymatrix@C=30pt
{ 
0\ar[r] & (\bigwedge^p\check\Lambda_y)^{\pi_1} \ar[r] &
 \bigwedge^p\check\Lambda_y \ar^(.38){(T^t_1-\id,...,T^t_r-\id)}[rr] & &
\bigwedge^p\check\Lambda_y\oplus...\oplus \bigwedge^p\check\Lambda_y,
}
$$
and since the labeled arrows are dual to each other we deduce the claimed properties about \eqref{eq-key-map}.
\end{proof}
We remark that the top horizontal map in \eqref{eq-key-map} is an isomorphism if and only if it is injective which in turn is equivalent to $H_0(U,\iota_*\bigwedge^p\Lambda)$ being torsion-free. 
For $U$ a product of symple models and a trivial factor as in Definition~\ref{definition-simple}, the space $H_0(U,\iota_*\Lambda)$ has torsion if and only if, for at least one $j$, the simplex $\triang_j$ is not isomorphic to the standard simplex\textsuperscript{\ref{fnstandard}}.

For any locally compact open subset $U\subset B$, we have by \cite[(50) on p.337]{Br97} a cap product
$$H_0(U,\iota_*\bigwedge^p\Lambda)\otimes H^0(U,\iota_*\bigwedge^p\check\Lambda)\stackrel\cap\lra H_0(U,\iota_*\bigwedge^p\Lambda\otimes\iota_*\bigwedge^p\check\Lambda).$$
\begin{lemma} 
\label{lemma-cap-zero}
There is a natural surjection $\op{tr}\colon H_0(U,\iota_*\bigwedge^p\Lambda\otimes\iota_*\bigwedge^p\check\Lambda)\ra\ZZ$ and the horizontal maps in \eqref{eq-key-map} agree with the maps induced from
the bilinear pairing $\op{tr}\circ\cap$.
\end{lemma} 
\begin{proof} 
For $y\in U\setminus\Delta$, we have the obvious dual pairing surjection $\frak{p}\colon\bigwedge^p\Lambda_y\otimes \bigwedge^p\check\Lambda_y\ra\ZZ$.
If $T\colon\bigwedge^p\Lambda_y\ra \bigwedge^p\Lambda_y$ is the monodromy transformation given by parallel transport along a $y$-based loop $\gamma$ in $U$ and $T^*$ is its dual, then $\frak{p}(v,w)=\frak{p}(T(v),T^*(w))$. In other words, $\im(T\otimes T^*-\id)\subseteq \ker(\frak{p})$.
We have $H_0(U,\iota_*\bigwedge^p\Lambda\otimes\iota_*\bigwedge^p\check\Lambda)=(\bigwedge^p\Lambda_y\otimes \bigwedge^p\check\Lambda_y) / K$ where $K$ is the subgroup generated by the images of endomorphisms of the form
$(T\otimes T^*-\id)$ for $T$ the monodromy of a $y$-based loop. 
Since $(\bigwedge^p\Lambda_y\otimes \bigwedge^p\check\Lambda_y) / K \sra (\bigwedge^p\Lambda_y\otimes \bigwedge^p\check\Lambda_y)/ \ker(\frak{p}) \cong\ZZ$, we verified the existence and surjectiveness of $\op{tr}$.

The cap product at chain level is defined in \cite[(49) on p.337]{Br97} as the straightforward map that takes an $i$-chain $\alpha$ with coefficients in $\shA$ to the $i$-chain $\alpha\otimes\beta$ for a given 
global section $\beta\in\Gamma(U,\shB)$.
We need to verify that the horizontal maps in \eqref{eq-key-map} derive from this form after applying $\op{tr}$. 
Let $\beta\in\Gamma(U, \bigwedge^p\check\Lambda)$ be given and
consider the commutative diagram with exact rows
$$ 
\resizebox{\textwidth}{!}{
\xymatrix@C=30pt
{ 
\bigwedge^p\Lambda_y\oplus...\oplus \bigwedge^p\Lambda_y \ar^(.6){\sum_i(T_i-\id)}[rr]\ar^{\cdot\otimes\beta}[d] && \bigwedge^p\Lambda_y \ar[r]\ar^{\cdot\otimes\beta}[d] & H_0(U,\iota_*\bigwedge^p\Lambda) \ar[r]\ar^{\cdot\cap\beta}[d] & 0\\
\bigwedge^p\Lambda_y\otimes \bigwedge^p\check\Lambda_y\oplus...\oplus \bigwedge^p\Lambda_y\otimes \bigwedge^p\check\Lambda_y \ar^(.6){\sum_i(T_i\otimes T_i^*-\id)}[rr] && \bigwedge^p\Lambda_y\otimes \bigwedge^p\check\Lambda_y \ar[r]\ar^{\frak{p}}[d] & H_0(U,\iota_*\bigwedge^p\Lambda\otimes \bigwedge^p\check\Lambda)\ar^{\op{tr}}@{-->}[ld]  \ar[r] & 0\\
&&\ZZ.
}
}
$$
The claim follows from the observation that passing $\alpha\in H_0(U,\iota_*\bigwedge^p\Lambda)$ through the diagram to the bottom $\ZZ$ agrees with evaluating it on $\beta$ after mapping it under the horizontal map in \eqref{eq-key-map} by the construction of this horizontal map in the proof of Lemma~\ref{lemma-sections}.
\end{proof}

\begin{lemma} 
\label{lem-tube-is-mfd}
The boundary of a tubular neighborhood of $\Delta\subset B$ in a symple integral affine manifold $B$ is a topological manifold.
\end{lemma}
\begin{proof} 
This follows from the fact that the boundary of the amoeba of a hyperplane is a topological manifold. 
\end{proof}

%===========================================================
%
%		section: First cohomology of punctured neighbourhoods
%
%===========================================================

\section{First cohomology of punctured neighborhoods}
Let $B$ be a symple affine $n$-manifold with singularities $\Delta$, $p\in B$ a point and $U\subset B$ a small open neighborhood of $p$ in $B$. 
We set $\Lambda_\QQ=\Lambda\otimes_\ZZ\QQ$. The goal of this section is to prove the following theorem which will be the main ingredient for the perfect pairing in the following section.
\begin{theorem} 
\label{H1-of-punctured-is-zero}
If $n\ge 3$ then
$H^1(U\setminus\{p\},\iota_*\Lambda_\QQ)=0.$ 
\end{theorem}

\begin{proof} 
If $p\in\partial B$, then $U\setminus \{p\}$ is contractible and the assertion directly follows from the fact that $\Delta$ meets $\partial B$ transversely, so assume $p\not\in\partial B$.
In view of Remark~\ref{rem-dir-sum}, it suffices to prove the theorem for $(B,\Delta)=(B_j\times\RR^c,Y_j\times \check Y_j\times\RR^c)$ where $(B_j,Y_j\times \check Y_j)$ is a symple model.
By Remark~\ref{wlog-p-origin}, it suffices to assume $p$ is the origin.
For $n=3$, the result is Lemma~\ref{3d-lem_removevertex} and for $n\ge 4$, this is Proposition~\ref{prop-H1-of-punctured-is-zero}.
\end{proof}

If $p\not\in\Delta$ then $\iota_*\Lambda_\QQ$ is a constant sheaf on $U\setminus \{p\}$. 
Since $U\setminus \{p\}$ is homotopic to $S^{n-1}$, Theorem~\ref{H1-of-punctured-is-zero} follows directly in this case. 
We thus assume from now on that $p\in\Delta$.
Furthermore, we will be able to prove the case $n\ge 4$ by a general machinery and then treat the case $n=3$ at the end of this section. The proof of the next statement will occupy most of this section.
\begin{proposition} 
\label{prop-H1-of-punctured-is-zero}
If $a,b>0$, $n=a+b+c\ge 4$, $(U,\Delta)$ is the product of a symple model $((\RR^{a})^*\times \RR^{b}),Y\times\check Y)$ with $\RR^c$ and $p\in\Delta$ the origin then $H^1(U\setminus\{p\},\iota_*\Lambda_\QQ)=0$.
\end{proposition}

In the situation of Proposition~\ref{prop-H1-of-punctured-is-zero}, consider the constant sheaf $G$ with stalk $\Gamma(U,\iota_*\Lambda_\QQ)$ on $U$ and set $Q = \iota_*\Lambda_\QQ/G$, so we have an exact sequence
\begin{equation} \label{punctured-ses-one}
0\ra G\ra \iota_*\Lambda_\QQ \ra Q\ra 0.
\end{equation}
By the particular form of the monodromy \eqref{eq-explicit-monodromy}, all stalks $Q_y$ for different $y\in U\setminus\Delta$ are identified unambiguously by parallel transport.
Let $\check G$ denote the constant sheaf on $U$ with stalk $Q_y$ for some $y\in U\setminus\Delta$. 
We have another exact sequence
\begin{equation} \label{punctured-ses-two}
0\ra Q \ra \check G \ra S\ra 0
\end{equation}
defining a constructible sheaf $S$ that is concentrated on $U\cap \Delta$.
For $n\ge 4$, since $U\setminus \{p\}$ is homotopic to $S^{n-1}$, we have
$$ H^1(U\setminus \{p\},G)=H^2(U\setminus \{p\},G)=0$$
and thus using the long exact sequences of cohomology for the short exact sequence \eqref{punctured-ses-one} yields
\begin{equation}
\label{h1lambdah1q}
H^1(U\setminus\{p\},\iota_*\Lambda_\QQ)=H^1(U\setminus\{p\},Q).
\end{equation}
Consider the long exact sequence of cohomology obtained from the short exact sequence \eqref{punctured-ses-two}. If $n\ge 3$ then 
$H^1(U\setminus \{p\},\check G)=0$,
so the sequence reads
\begin{equation}
\label{lesofsecondses}
\Gamma(U\setminus \{p\},Q)\ra \Gamma(U\setminus \{p\},\check G) \ra \Gamma(U\setminus \{p\},S)\ra H^1(U\setminus \{p\},Q)\ra 0
\end{equation}
and gives the following result.
\begin{proposition} 
\label{prop-surjective-to-S}
If $\dim B\ge 4$ then 
$$H^1(U\setminus\{p\},\iota_*\Lambda_\QQ)=0\quad \iff\quad  \Gamma(U\setminus \{p\},\check G) \ra \Gamma(U\setminus \{p\},S) \hbox{ is surjective}.$$
\end{proposition}
We abuse notation by using $\check G$ not only for the constant sheaf but also for its stalk at $p$.
\begin{lemma}
\label{lemma-coho-S-on-U}
$H^k(U,S)=0$ unless $k=0$ where it equals $\check G$.
\end{lemma}

\begin{proof} 
We use that $\Delta$ is conical at $p$, i.e.~$\lambda\Delta\subset\Delta$ for $\lambda\in(0,1]$. 
Furthermore, $\iota_*\Lambda$, $Q$ and $S$ respect this conical structure, i.e.~for $f_\lambda\colon U\ra U, x\mapsto \lambda x$, we naturally identify $f_\lambda^{-1}S=S$. This implies
$H^k(U,S)=H^k(f^{-1}_\lambda U,S)$ for all $\lambda$ and thus $H^k(U,S)$ coincides with the stalk of the cohomology presheaf $V\mapsto H^k(V,S)$ at $p$. 
However this is zero for $k>0$ as it is $(R^k\id_*(S))_p$ which is obtained by taking cohomology on an injective resolution after applying $\id$. 
To determine the $k=0$ case, note that, by a similar conical argument and by the definition of $G$, we have 
$\Gamma(U,G)=\Gamma(U,\iota_*\Lambda_\QQ)=(\iota_*\Lambda_\QQ)_p$, so $Q_p=0$. Exactness of \eqref{punctured-ses-two} then gives $\Gamma(U,S)=S_p=\check G_p=\check G$.
\end{proof}

\begin{lemma} 
\label{lemma-local-sections} 
We have $H^k(U,\iota_*\Lambda)=0$ for $k>0$.
\end{lemma}
\begin{proof}
The statement follows from the conical shape of $\Delta$ as explained in the proof of Lemma~\ref{lemma-coho-S-on-U}.
\end{proof}

The combination of Proposition~\ref{prop-surjective-to-S} and Lemma~\ref{lemma-coho-S-on-U} (and the fact that $\Gamma(U\setminus \{p\},\check G)=\check G$) gives the following.
\begin{proposition} 
\label{prop-surjective-to-S-2}
If $\dim B\ge 4$ then 
$$H^1(U\setminus\{p\},\iota_*\Lambda_\QQ)=0\quad \iff\quad  \Gamma(U,S) \ra \Gamma(U\setminus \{p\},S) \hbox{ is surjective}.$$
\end{proposition}

\begin{lemma} 
\label{lemma-Delta-described}
The stalk of $Q$ at a point $(x,y,z)\in Y\times\check Y\times \RR^c$ only depends on $y$ and is given by
$$\Hom(T_{\check\triang} / T_{\tau_y},\QQ)$$
where 
\begin{itemize}
\item $\tau_y$ is the face of $\check\triang$ dual to the minimal stratum of $\check Y$ containing $y$,
\item $T_{\sigma}$ denotes the tangent space of $\sigma$ over $\QQ$ for $\sigma=\tau_y$ or $\sigma=\check \triang$.
\end{itemize}
This entirely describes $Q$: if $y'$ is contained in a smaller stratum than $y$ then $T_{\tau_y}\subset T_{\tau_{y'}}$, so
$T_{\check\triang} / T_{\tau_{y}}\subset T_{\check\triang} / T_{\tau_{y'}}$
 and
$$\Hom(T_{\check\triang} / T_{\tau_{y'}},\QQ)\subset \Hom(T_{\check\triang} / T_{\tau_y},\QQ)$$
is the generization map $Q_{y'}\to Q_{y}$.
\end{lemma}
\begin{proof} 
The description follows directly from the definition of $Q$ and \eqref{eq-explicit-monodromy}.
\end{proof}

Consequently, at the stalk at $(x,y,z)\in Y\times\check Y\times \RR^c$, the sequence \eqref{punctured-ses-two} reads
$$0 \ra \Hom(T_{\check\triang} / T_{\tau_y},\QQ) \ra \Hom(T_{\check\triang} ,\QQ) \ra \Hom(T_{\tau_y},\QQ) \ra 0$$
and we conclude the next lemma.
\begin{lemma} 
\label{sheafS-described}
The stalk of $S$ at a point $(x,y,z)\in  Y\times\check Y\times \RR^c$ is given by $\Hom(T_{\tau_y},\QQ)$.
\end{lemma}
In particular $S$ is the pullback of a sheaf (that we also call $S$) under the projection $Y\times\check Y\times \RR^c\ra \check Y$.

This lemma will enable us to compute the cohomology of $S$ on $U\setminus\{p\}$. The following complex will appear in the computation:
\begin{equation}
\label{complex-1}
0\ra \bigoplus_{{\tau\subset\check\triang}\atop{\dim \tau =0}} T_\tau \ra \bigoplus_{{\tau\subset\check\triang}\atop{\dim \tau =1}} T_\tau  \ra ...\ra T_{\check \triang}\ra 0
\end{equation}
where the maps are given by inclusions weighted by a sign that is obtained from an ordering of the vertices of $\check \triang$ (as usual in algebraic topology).
Set $C^i =\bigoplus_{{\tau\subset\check\triang}\atop{\dim \tau =i}} T_\tau$ and note that $C^0=0$.
\begin{lemma} 
We have
$H^k(C^\bullet)=0$ unless $k=1$ where it has rank one.
\end{lemma} 
\begin{proof} 
First note that any two complexes \eqref{complex-1} for different lattice simplices $\check\triang$ of the same dimension are isomorphic because taking tangent spaces forgets the information about the lattice and over $\QQ$ all simplices of the same dimension are isomorphic. We may thus assume $\check\triang$ is a standard simplex.
Since projective space $\PP^r$ is a toric variety, its cohomology for every torus-invariant sheaf is graded by the character lattice $\ZZ^r$. 
We are interested in the coherent sheaf of K\"ahler differentials $\Omega_{\PP^r}$.
By the finiteness of cohomology ranks, all non-zero cohomology of the \v{C}ech complex of the standard open cover of torus invariant charts for $\Omega_{\PP^r}$ is concentrated in $\ZZ^r$-degree zero, see \cite{Da78}.
The only non-trivial cohomology is $H^1(\PP^r,\Omega_{\PP^r})\cong\CC$.
The grading degree zero part of the \v{C}ech complex for $\Omega_{\PP^r}$ is identified with the complex \eqref{complex-1} after tensoring the latter with $\CC$, see \cite[Proposition 4.3]{Da78}.
This implies the assertion. 
\end{proof}
We also need a complex closely related to \eqref{complex-1}, namely
\begin{equation}
\label{complex-2}
0\ra \bigoplus_{{\tau\subset\triang}\atop{\dim \tau =0}} \QQ \ra \bigoplus_{{\tau\subset\triang}\atop{\dim \tau =1}} \QQ  \ra ...\ra \bigoplus_{\tau={\triang}} \QQ\ra 0
\end{equation}
 and since it computes $H^i(\triang,\QQ)$, denoting $D^i =\bigoplus_{{\tau\subset\triang}\atop{\dim \tau =i}} \QQ$, we get the following.
\begin{lemma} 
$H^k(D^\bullet)=0$ unless $k=0$ where it has rank one.
\end{lemma} 
We take the duals $\check C^i = \Hom(C^{b-i},\QQ)$ and $\check D^i = \Hom(D^{a-i},\QQ)$ and obtain from the previous lemmata the following.
\begin{lemma} 
\begin{enumerate} 
\item $H^k(\check C^\bullet)=0$ unless $k=\dim\check\triang-1$ where it has rank one.
\item $H^k(\check D^\bullet)=0$ unless $k=\dim\triang$ where it has rank one.
\end{enumerate}
\end{lemma} 
Finally, we consider the truncated dual of $\check D^\bullet$,
$$\check{\bar{D}}^j = \left\{\begin{array}{ll}\Hom(D^{a-j},\QQ),& a\neq j\\ 0,& a=j.\end{array}\right.$$
for which we conclude the following.
\begin{lemma} 
$H^k(\check{\bar{D}}^\bullet)= 0$ unless $k=\dim\triang-1$ where it has rank $|\op{vertices}(\triang)|-1=\dim \triang$.
\end{lemma} 
By tensoring the cohomologies of the factors, we obtain the cohomology of the tensor product:
\begin{lemma} 
\label{CD-lemma}
$H^k(\check{\bar{D}}^\bullet\otimes \check C^\bullet)= 0$ unless $k=\dim\triang+\dim\check\triang-2$ where it has rank $\dim \triang$.
\end{lemma} 

Using Proposition~\ref{prop-surjective-to-S}, we obtain the special case $c=0$ of Proposition~\ref{prop-H1-of-punctured-is-zero} from the $k=0$ part of (1) of the following result.

\begin{proposition} 
\begin{enumerate}
\item
If $\dim\triang+\dim\check\triang\ge 4$ then
$$H^k(Y\times \check Y\setminus\{(0,0)\}, S)=\left\{\begin{array}{ll} \check G & k=0\\ H & k=\dim\triang+\dim\check\triang-3  \\0 &\hbox{otherwise} \end{array}\right.$$
where $H$ is defined by the equality and is of rank $\dim \triang$.
\item
If $\dim\triang+\dim\check\triang=3$ then $H^k(Y\times \check Y\setminus\{(0,0)\}, S)=0$ for $k>0$ and we have an exact sequence
$$0\ra \check G\ra H^0(Y\times \check Y\setminus\{(0,0)\}, S)\ra H\ra 0$$
where again $H$ is defined by the sequence and is of rank $\dim \triang$.
\item
Finally for $\dim\triang+\dim\check\triang=2$, $Y\times\check Y=\{0,0\}$, so by triviality $H^k(Y\times \check Y\setminus\{(0,0)\}, S)=0$ for all $k$.
\end{enumerate}
\label{prop-punctured-Y-checkY-computed}
\end{proposition}
\begin{proof}
We use an open cover of $(Y\times\check Y)\setminus \{(0,0)\}$. 
Note that strata of $Y$ are indexed by faces of $\triang$ of dimension at least one and similarly strata of $\check Y$ are indexed by faces of $\check \triang$ of dimension at least one.
For $\tau\in\triang$ a face of dimension at least one, let $Y_\tau\subset Y$ be the open subset given by the the open star of the relative interior of the stratum of $Y$ corresponding to $\tau$, i.e.~$Y_\tau$ is the union of the relative interiors of the strata of $Y$ corresponding to the faces of $\triang$ containing $\tau$. 
We similarly define open subsets $\check Y_{\check\tau}\subset \check Y$ indexed by faces $\check\tau\subset\check\triang$ of dimension at least one.
The point of this is that
$$\frak{U}=\{Y_\tau\times Y_{\check\tau}\mid (\dim\tau,\dim\check\tau)\in \{(\dim\triang,\dim\check\triang-1),(\dim\triang-1,\dim\check\triang)\}\}$$
is an open cover of $(Y\times\check Y)\setminus \{(0,0)\}$.
We have 
$$(Y_{\tau_1}\times Y_{\check\tau_1})\cap (Y_{\tau_2}\times Y_{\check\tau_2})\cap ...\cap (Y_{\tau_k}\times Y_{\check\tau_k}) = (Y_{\tau_1\cap...\cap\tau_k}\times Y_{\check\tau_1\cap...\cap\check\tau_k})$$
and therefore identify the \v{C}ech complex $\check C^\bullet(\frak{U},S)$ of $S$ with respect to $\frak{U}$ with the complex
$$
\begin{array}{rcl}
\check C^k(\frak{U},S) 
&=& \bigoplus_{{(\tau_0,\check\tau_0),...,(\tau_k,\check\tau_k)}\atop{\hbox{\tiny pairwise distinct}}} \Gamma(Y_{\tau_1\cap...\cap\tau_k}\times Y_{\check\tau_1\cap...\cap\check\tau_k},S)\\
&=& \bigoplus_{I_k}\Gamma(Y_{\tau}\times Y_{\check\tau},S)\\
&=& \bigoplus_{I_k} \QQ\otimes_\QQ\Gamma(Y_{\check\tau},S)\\
&\overset{\hbox{\tiny Lemma~\ref{sheafS-described}}}{=}& \bigoplus_{I_k} \QQ\otimes_\QQ \Hom(T_{\check\tau},\QQ)\\
&=&  \bigoplus_{i+j=k+1}\check{\bar{D}}^i\otimes \check C^j
\end{array}
$$
where $I_k$ denotes the set of pairs $\{(\tau,\check\tau)\,\mid\,\tau\subset\triang,\check\tau\subset\check\triang, \codim\tau+\codim\check\tau=k+1\}$. 
Note that we need the truncation $\check{\bar{D}}^\bullet$ of $\check{{D}}^\bullet$ because $Y$ has no strata indexed by vertices of $\triang$. 
For $\check Y$ and $\check C^\bullet$, this is already taken care of because $C^0=0$.
Note that the zero'th term $\check{\bar{D}}^0\otimes \check C^0$ is not part of the above complex but all other terms of $\check{\bar{D}}^\bullet\otimes \check C^\bullet$ are.
Acknowledging the shift $k+1\leadsto k$, the assertion now follows from Lemma~\ref{CD-lemma} by noting that $\check{\bar{D}}^0\otimes \check C^0 = \Hom(T_{\check\Delta_1},\QQ)=\check G$.
\end{proof}

\begin{proof}[Proof of Proposition~\ref{prop-H1-of-punctured-is-zero}] We have $n\ge 4$.
As said before, Proposition~\ref{prop-surjective-to-S} and Proposition~\ref{prop-punctured-Y-checkY-computed} gives the $c=0$ case, so assume $c>0$. 
Again by Proposition~\ref{prop-surjective-to-S}, we are done if we show that $H^0(U\setminus \{p\},S)=\check G$.
Since $S$ has support on $\Delta$, we have 
$$H^0(U\setminus \{p\},S)=H^0((Y\times\check Y\times\RR^c)\setminus \{(0,0,0)\},S)$$
and we are going to compute this via the open cover consisting of the two sets $(Y\times\check Y\setminus\{(0,0)\})\times\RR^c$ and $Y\times\check Y\times(\RR^c\setminus\{0\})$.
Recall from Lemma~\ref{sheafS-described} that $S$ is a sheaf pulled back from the $\check Y$-factor. Mayer-Vietoris yields the Cartesian diagram
\begin{equation}
\label{MV-Cartesian}
\begin{gathered}
\xymatrix@C=30pt
{ 
H^0(U\setminus \{p\},S) \ar[rr]\ar[d]&&  H^0(\RR^c,\QQ)\otimes H^0(Y\times\check Y\setminus\{(0,0)\},S) \ar^{restriction\otimes\id}[d] \\
H^0(\RR^c\setminus\{0\},\QQ)\otimes\check G\ar^(.38){id\otimes restriction}[rr] && H^0(\RR^c\setminus\{0\},\QQ)\otimes H^0(Y\times\check Y\setminus\{(0,0)\},S).
}
\end{gathered}
\end{equation}
We now need to distinguish two cases. 

If $a=b=1$, then by (3) of Proposition~\ref{prop-punctured-Y-checkY-computed} the right column of \eqref{MV-Cartesian} is zero. 
Now $n\ge 4$ implies $c\ge 2$, hence $H^0(\RR^c\setminus\{0\},\QQ)\cong\QQ$ and therefore $H^0(U\setminus \{p\},S)=\check G$, so we are done.

If $a+b>2$ then we are in either (1) or (2) of Proposition~\ref{prop-punctured-Y-checkY-computed} both giving that the restriction
$$ H^0(Y\times\check Y,S) \ra H^0(Y\times\check Y\setminus\{(0,0)\},S)$$
is injective. The restriction $H^0(\RR^c,\QQ) \ra H^0(\RR^c\setminus\{0\},\QQ)$ is an isomorphism if $c>1$ and is the diagonal $\QQ\ra\QQ^2$ if $c=1$. 
Either way, \eqref{MV-Cartesian} implies $H^0(U\setminus \{p\},S)=\check G$, so we are done with proving Proposition~\ref{prop-H1-of-punctured-is-zero}.
\end{proof}

We finally treat the three-dimensional case.

\begin{lemma} 
\label{3d-lem_removevertex} 
Let $B$ be symple threefold, $p\in B$ and $U\subset B$ be a small open neighborhood of $p$ then
$H^1(U\setminus \{p\}, \iota_*\Lambda_\QQ)=0.$
\end{lemma}
\begin{proof} 
If $p\not\in\Delta$ then $\iota_*\Lambda_\QQ$ is a local system with stalk $\QQ^3$ on $U$ and the assertion follows from $H^1(S^2,\QQ)=0$ since $U\setminus \{p\}$ retracts to an $S^2$. 
The case $p\in\partial B$ is trivial. Note that $\Delta$ is a trivalent graph. 

First assume that $p\in\Delta$ is a (trivalent) vertex.
Let $e_1,e_2,e_3$ denote the three components of $(\Delta\cap U)\setminus \{p\})$. 
We use a \v{C}ech computation with a cover featuring three contractible open sets $U_1,U_2,U_3$ covering $U\setminus\{p\}$ where $U_j$ contains $e_j$ and is disjoint from $e_k$ for $k\neq j$. Such open sets can be chosen with the extra property that
$U_j\cap U_k$ is contractible for every $j,k$ and $U_1\cap U_2\cap U_3$ has two connected components each of which is contractible. 
Let $C^m=\bigoplus_{j_0<...<j_m} \Gamma(U_{j_0}\cap...\cap U_{i_m},\iota_*\Lambda_\QQ)$ denote the terms, so we have a \v{C}ech complex
$$0\ra C^0\stackrel{d_0}{\lra} C^1\stackrel{d_1}{\lra} C^2\ra 0$$
where $\rk C^0=3\cdot 2=6$, $\rk C^1=3\cdot 3=9$ and $\rk C^2=2\cdot 3=6$.
We want to show that $\rk\ker d_1=\rk d_0$. Note that $d_1$ is given by the matrix
\begin{equation} \label{matrix-3d}
\begin{pmatrix} 
\id &-\id&\id\\
\id &-T_1&T_2\\
\end{pmatrix}
\end{equation}
where the columns are given by bases of the sections of $\iota_*\Lambda_\QQ$ on $U_1\cap U_2$, $U_1\cap U_3$, $U_2\cap U_3$, the two rows correspond to the two components of $U_1\cap U_2\cap U_3$ and $T_i$ denotes the monodromy transformation around $e_i$. Using Gaussian elimination to get rid of the first column, we find that the kernel of \eqref{matrix-3d} has the same rank as the kernel of the $6\times 3$-matrix
$$
\begin{pmatrix} 
\id-T_1&T_2-\id
\end{pmatrix}
$$
which is $2\cdot 2=4$-dimensional in the case $a=2$ (monodromy invariant plane at $p$) and $5$-dimensional in the case $a=1$ (monodromy invariant line at $p$). 
On the other hand, since the restriction $H^0(U, \iota_*\Lambda)\ra H^0(U\setminus \{p\}, \iota_*\Lambda)$ is an isomorphism, $\dim\ker(d_0)=a$, so the rank of $d_0$ is $6-a$ and we conclude the assertion.

It remains to check the case when $p\in\Delta$ is bivalent or contained in an edge of $\Delta$. Let $e_1,e_2$ be the outgoing legs of $\Delta$ at $p$ and take the \v{C}ech cover of $U\setminus p$ consisting of the two open sets $U_1=U\setminus e_1$, $U_2=U\setminus e_2$ then $H^k(U_j,\iota_*\Lambda)=0$ for $k>0$ and $j=1,2$. Moreover, $H^0(U_1\cap U_2,\iota_*\Lambda)=H^0(U_j,\iota_*\Lambda)=\Gamma(U,\iota_*\Lambda)$ for $j=1,2$ which implies the result.
\end{proof}

%===========================================================
%
%		section: Perfect paring in degree one
%
%===========================================================

%\section{Homology-cohomology-pairing on graphs}
%let $G$ be the topological realization of a graph with its natural structure as a $\triang$-complex. 
%Let $\Lambda$ and $\Lambda'$ be acyclic sheaves of Abelian groups on $G$ together with a map $\Lambda\otimes\Lambda'\ra\ZZ$.

\section{The pairing and its perfectness in degree one}
This section is occupied with proving the following result.

\begin{theorem} 
\label{pairingthm}
Let $(B,\P)$ be an integral affine manifold with singularities (as defined in \S\ref{sec-simple}) with $\iota\colon B\setminus\Delta\hra B$ the inclusion of the regular locus. 
For each $p,q\ge 0$, there is a commutative diagram
$$
\xymatrix@C=30pt
{ 
H_q(B,\iota_*\bigwedge^p\Lambda) \ar[r] & \Hom(H^q(B,\iota_*\bigwedge^p\check\Lambda),\ZZ) \\
H_q(B\setminus\Delta,\bigwedge^p\Lambda) \ar[r] \ar[u] & \Hom(H^q(B\setminus\Delta,\bigwedge^p\check\Lambda),\ZZ) \ar[u] 
}
$$
with vertical maps given by $\iota$. 
The horizontal maps generalize the horizontal maps in \eqref{eq-key-map}.
A similar diagram exists if we replace $\Lambda,\check\Lambda$ by $\Lambda_\QQ,\check\Lambda_\QQ$ and there is a natural map between these two diagrams giving a commutative cube.
If $B$ is symple, in the square diagram with $\Lambda_\QQ,\check\Lambda_\QQ$ the bottom horizontal map is an isomorphism and, if in addition $p=1$ and $q\le 1$, also the top horizontal map is an isomorphism.
\end{theorem}
We remark that the top horizontal map in the theorem defines the pairing \eqref{eq-gen-pairing} and the statement about the isomorphism over $\QQ$ for $p=q=1$ proves Theorem~\ref{main-thm-intro}. See Theorem~\ref{thm-pairings-agree} and Corollary~\ref{pairing-is-cap-product} below for the connection to the cap and cup product.

\begin{cexample} 
\label{counterexample-conifold}
A point $p$ in an affine 3-manifold $B$ with singularities $\Delta$ is a \emph{conifold point} if a neighborhood $U$ of $p$ is homeomorphic to a local model defined similar to Definition~\ref{def-simple-model} where either $\triang$ or $\check\triang$ is a unit square and the other is a unit interval, see~\cite{CBM09,RS}. 
Consequently, $\Delta$ is a four-valent graph locally at $p$ and we require $p$ to be the four-valent vertex. 
By the same proof as that for Lemma~\ref{lemma-local-sections}, we get $H^1(U,\iota_*\Lambda)=H^1(U,\iota_*\check\Lambda)=0$. 
On the other hand, $\rk H_1(U,\iota_*\Lambda)=\rk H_1(U,\iota_*\check\Lambda)=1$ can be deduced via the arguments\footnote{Use that $\partial U$ is symple and apply \eqref{induction-iso} and \eqref{iso-in-front}.} in the proof of Proposition~\ref{hypothesis-for-lemma} combined with the knowledge that 
$\rk H^1(\partial U,\iota_*\Lambda)=\rk H^1(\partial U,\iota_*\check\Lambda)=1$ which in turn can be gained from an easy \v{C}ech computation. 
The statement of the top horizontal map in Theorem~\ref{pairingthm} being an isomorphism for $p=q=1$ over $\QQ$ thus fails for $B=U$ (because $B$ is not symple).
\end{cexample}

One would usually deduce perfectness from the sheaf level, i.e.~use the isomorphism $\Lambda\ra\Hom(\check\Lambda,\ZZ)$. The difficulty is that taking $\iota_*$ messes this up. 
Indeed, while there is a sheaf map $\varphi\colon \iota_*\Lambda\ra \Hom(\iota_*\check\Lambda,\ZZ)$, it is \emph{not} an isomorphism, e.g.~it is the zero map at a focus-focus-point. 
To get around this problem, we replace $\varphi$ by the top horizontal map in \eqref{eq-key-map} of Lemma~\ref{lemma-sections}, i.e.~we locally work with the \emph{homology} of closed stars rather than a map of sheaves, see Lemma~\ref{iso-via-clstars} below. 

In the symple case, \eqref{eq-gen-pairing} in degrees other than $p=q=1$ is presumably also perfect over $\QQ$ but I found this hard to prove. 

Here is how we are going to proceed: as the first step in Lemma~\ref{iso-via-clstars}, we introduce a general tool to produce a map from homology into the dual of cohomology. The Lemma doesn't seem to do this as stated but we are going to apply it for the situation where the target of the map $\varphi$ is the homology of the dual of a \v{C}ech complex for $\iota_*\bigwedge^p\check\Lambda$. If we make sure to satisfy the assumptions of the Lemma, most importantly, $H_1(\P_\tau,A)=0$, then the Lemma will imply Theorem~\ref{pairingthm}. 
The key ingredient about the vanishing first homology group is then provided by Proposition~\ref{hypothesis-for-lemma}. The proof mostly does a diagram chase where in turn the key ingredient is the vanishing of the first cohomology of punctured neighbourhoods that we have shown in Theorem~\ref{H1-of-punctured-is-zero} if $\dim B\ge 3$ and the two-dimensional case can be treated separately.

\begin{definition} 
\label{def-chain-complex}
Let $\P$ be a locally finite simplicial complex together with a total order of its vertices and let $A$ a contra-variant functor from $\P$ into groups, so a group $A_\tau$ for every $\tau\in\P$ with maps $A_\tau\ra A_{\tau'}$ whenever $\tau'\subset\tau$.
We call $(\P,A)$ a \emph{chain complex}. Let $\P^{[i]}\subset \P$ denote the subset of $i$-dimensional simplices.
The homology complex
$C_\bullet(\P,A)$ with terms
$C_i(\P,A)=\bigoplus_{\tau\in\P^{[i]}} A_\tau$
and differential $d_i\colon C_i(\P,A)\ra C_{i-1}(\P,A)$ is defined in the usual way by the summing over the maps $A_\tau\ra A_{\tau'}$ weighted by orientation for $\tau'\subset\tau$ with $\tau'\in\P^{[i-1]}, \tau\in\P^{[i]}$ using the ordering of vertices.
We denote the homology by $H_i(\P,A)$. 
\end{definition}

In view of Definition~\ref{def-constructible}, if $\Lambda$ is a $\P$-constructible sheaf on the topological realization $|\P|$ of $\P$ and the functor $A$ is defined by $\tau\mapsto \Gamma(U_\tau,\Lambda)$ for $U_\tau$ a small neighborhood of $\tau$ in $|\P|$, then
we obtain a natural isomorphism $H^\P_i(|\P|,\Lambda)=H_i(\P,A)$.  
If $\Lambda$ is $\P$-acyclic, Theorem~\ref{simplicial=singular} gives
\begin{equation}
\label{eq-iso-homology-chain-complex}
H_i(|\P|,\Lambda)=H_i(\P,A).
\end{equation} 

\begin{definition} 
\label{def-clstar}
Given a locally finite simplicial complex $\P$ and $\tau\in\P$, we define the \emph{closed star} of $\tau$ to be the finite simplicial complex given by
$$\P_\tau=\{\omega\in\P\mid \omega \hbox{ is a face of a simplex that contains }\tau\}.$$
\end{definition}

\begin{lemma} 
\label{iso-via-clstars}
Let $(\P,A)$ and $(\P,B)$ be chain complexes and assume $|\P_\tau|$ is contractible for all $\tau\in\P$.
Assume that we have, for each $\tau\in\P$, a map
$$\varphi_\tau\colon  H_0(\P_\tau,A) \ra B_\tau $$
which is compatible with inclusions. Then the set of $\varphi_\tau$ induces a composition
$$\varphi\colon H_i(\P,A)\ra H_i(\P,A')\ra H_i(\P,B)$$
for $A'$ the functor $\tau\mapsto H_0(\P_\tau,A)$ and for each $i\ge 0$.
If in addition each $\varphi_\tau$ is an isomorphism and
$$H_i(\P_\tau,A)=0$$
for $1\le i\le k$ then $\varphi$ is an isomorphism in degrees $i\le k$.
\end{lemma}

\begin{proof} 
Consider the double complex
\begin{equation}
\begin{gathered}
\label{eq-double-complex-homology}
\resizebox{0.9\textwidth}{!}{
\xymatrix@C=30pt
{ 
\vdots\ar[d]&\vdots\ar[d]&\vdots\ar[d]&\\
\bigoplus_{\tau\in\P^{[0]}} C_2(\P_\tau,A)\ar[d] & \bigoplus_{\tau\in\P^{[1]}} C_2(\P_\tau,A) \ar[l]\ar[d] & \bigoplus_{\tau\in\P^{[2]}} C_2(\P_\tau,A) \ar[l]\ar[d] &\ar[l]\dots\\
\bigoplus_{\tau\in\P^{[0]}} C_1(\P_\tau,A)\ar[d] & \bigoplus_{\tau\in\P^{[1]}} C_1(\P_\tau,A) \ar[l]\ar[d] & \bigoplus_{\tau\in\P^{[2]}} C_1(\P_\tau,A) \ar[l]\ar[d] &\ar[l]\dots\\
\bigoplus_{\tau\in\P^{[0]}} C_0(\P_\tau,A) & \bigoplus_{\tau\in\P^{[1]}} C_0(\P_\tau,A) \ar[l] & \bigoplus_{\tau\in\P^{[2]}} C_0(\P_\tau,A) \ar[l] & \ar[l]\dots
}
}
\end{gathered}
\end{equation}
where the vertical maps are the homology differentials and the horizontal maps are given by the injections $C_i(\P_{\tau},A)\subset C_i(\P_{\tau'},A)$ for $\tau'\subset\tau$ weighted by sign. 
We may view the double complex as the chain complex for the functor that associates to $\tau\in\P$ the summand $C_\bullet(\P_\tau,A)$ of the corresponding column.

The collection of $\varphi_\tau$ induces a map from the bottom row of the double complex to the complex $C_\bullet(\P,B)$. 
Moreover, if $T_\bullet$ denotes the total complex of the double complex, the described map constitutes a map of complexes $\varphi_1\colon T_\bullet\ra C_\bullet(\P,B)$.

We next look at the homology of the rows in the double complex. 
For fixed $\tau$, gathering all terms involving $A_\tau$, we obtain the homology complex of the closed star of $\tau$ with constant coefficients $A_\tau$. 
Since the closed star is contractible, the rows have homology concentrated in the left-most columns and together form the complex $C_\bullet(\P,A)$ in this column. 
We thus obtain another map of complexes $\varphi_2\colon T_\bullet\ra C_\bullet(\P,A)$ which a quasi-isomorphism. We now obtain the desired map $\varphi$
as the composition of $\varphi=\bar\varphi_1\circ\bar\varphi_2^{-1}$ where $\bar\varphi_1,\bar\varphi_2$ denote the map induced on homology by $\varphi_1,\varphi_2$ respectively.

The extra assumption that each $\varphi_\tau$ is an isomorphism and that $H_i(\P_\tau,A)=0$ for $1\le i\le k$ implies that $\bar\varphi_1$ is an isomorphism in degrees $\le k$ and thus implies the claim. 
\end{proof}

\begin{remark}
\label{remark-compatibility-subcomplex}
The assignment of $\varphi$ in Lemma~\ref{iso-via-clstars} is functorial, i.e.~$\varphi_{BC}\circ \varphi_{AB}=\varphi_{AC}$ when given another map of chain complexes $(\P,B)\ra (\P,C)$ and requiring $H_i(\P_\tau,B)=0$ for $1\le i\le k$.
The assignment of $\varphi$ is furthermore compatible with restriction to a subcomplex: let $\P'\subset\P$ be simplicial subcomplex whose closed stars are also contractible. 
Let $A',B'$ be the functors induced via restricting $A,B$ to $\P'$. Assume $H_i(\P'_\tau,A')=0$ for $1\le i\le k$. Then the following diagram commutes
$$
\xymatrix@C=30pt
{ 
H_i(\P',A')\ar^{\varphi'}[r]\ar[d] & H_i(\P',B')\ar[d] \\
H_i(\P,A)\ar^{\varphi}[r] & H_i(\P,B).
}
$$
\end{remark}

\begin{example}
\label{ex-pair-local-sys}
Assume $\shP$ is a locally finite simplicial complex with a total order of its vertices and whose topological realization is a topological manifold $M$. Assume that $|\shP_\tau|$ is contractible for each $\tau\in\shP$.
Let $L$ be a locally constant sheaf with stalks free finitely generated $\ZZ$-modules. 
Let $\check L=\Hom(L,\ZZ)$ denote the dual local system and let $W_\tau$ be the open star of $\tau$, i.e.~the interior of the closed star $|\shP_\tau|$.
The open sets $W_v$ with $v$ running over the vertices of $\shP$ give an open cover. 
For $v_1,...,v_k$ vertices, note that $W_{v_1}\cap...\cap W_{v_1}=W_\tau$  if $v_1,...,v_k$ form the simplex $\tau$ and $W_{v_1}\cap...\cap W_{v_1}=\emptyset$ otherwise.
Applying $\Hom(\cdot,\ZZ)$ to the \v{C}ech complex 
$$ \bigoplus_{\tau\in\P^{[0]}} \Gamma(W_\tau,\check L)\ra \bigoplus_{\tau\in\P^{[1]}} \Gamma(W_\tau,\check L)\ra \bigoplus_{\tau\in\P^{[2]}} \Gamma(W_\tau,\check L)\ra ...$$
yields the homology complex for the functor $B$ that sends $\tau\mapsto \Hom(\Gamma(W_\tau,\check L),\ZZ)$ in the sense of Definition~\ref{def-chain-complex}.
Consider the functor $A\colon  \tau\mapsto \Gamma(\tau, L)$, so we have an isomorphism of functors $A\ra B$, given by the tautological map $\Gamma(\tau, L)\ra \Hom(\Gamma(W_\tau,\check L),\ZZ)$.
The induced isomorphism in homology 
$\varphi\colon H_i(\shP,A)\ra H_i(\shP,B)$
agrees with the one from Lemma~\ref{iso-via-clstars}: indeed, $H_i(\P_\tau,A)=0$ holds for all $i>0$ since $L$ is constant on $\P_\tau$ and $\P_\tau$ contractible. 
Also $H_0(\P_\tau,A)=A_\tau$, so the map $\varphi_\tau$ in the Lemma is just the isomorphism $A_\tau\ra B_\tau$. 

The map $\varphi$ can be precomposed with the isomorphism \eqref{eq-iso-homology-chain-complex} using that $L$ is $\P$-acyclic.
The resulting isomorphism then may be postcomposed with the natural map $H_i(\shP,B)\ra \Hom(H^i(M,\check L),\ZZ)$, so that the entire composition yields the natural map $\alpha$ in the short exact sequence
 $$ 0\ra \operatorname{Ext}_\ZZ^1(H^{i+1}(M,\check L),\ZZ)\ra H_i(M,L)\stackrel{\alpha}\lra \Hom(H^i(M,\check L),\ZZ) \ra 0$$ 
and $\alpha$ is equivalent to the natural pairing
\begin{equation}
H_i(M, L)\otimes H^i(M,\check L)\ra\ZZ.
\label{eq-pairing}
\end{equation}
See \cite[Theorem~10.5]{Br97} and \cite[\S 12.2.1]{Cu13} for similar versions.
Explicitly, if $\beta=\sum_{\tau\in\shP^{[i]}} a_\tau$ for $a_\tau\in \Gamma(\tau,L)$ with only finitely many $a_\tau$ non-zero is an $i$-cycle representing an element in $H_i(M,L)$ and $s=(b_\tau)_{\tau\in\shP^{[i]}}$ with $b_\tau\in \Gamma(W_\tau,\check L)$ is a $i$-\v{C}ech-cocycle representing an element in $H^i(M,\check L)$ then the pairing of $\beta$ and $s$ under \eqref{eq-pairing} is explicitly given by
\begin{equation}
\langle\beta,s\rangle :=\sum_\tau \langle a_\tau,b_\tau\rangle
\end{equation}
where $\langle \, ,\,\rangle$ on the right hand side denotes the pairing $L\otimes\check L\ra \ZZ$.
\end{example}

\begin{lemma} 
\label{lemma-homotope-inside}
Let $B$ be a symple integral affine $n$-manifold with singularities $\Delta$ together with a triangulation $\P$ so that $\Delta$ is a simplicial subcomplex; in other words, $\P$ is a refinement of the natural stratification of $B$ given by the strata of $\Delta$. 
Let $|\P_\tau|\subset B$ denote the closed star of a simplex $\tau$ and $W_\tau:=\Int(|\P_\tau|)$ its interior. 
After replacing $\P$ with a barycentric subdivision if needed, we may assume each $|\P_\tau|$ strongly retracts to a point in its interior.
The natural map $H_i(W_\tau,\iota_*\bigwedge^p\Lambda)\ra H_i(|\P_\tau|,\iota_*\bigwedge^p\Lambda)$ is an isomorphism for each $i$. 
A similar statement holds for $\check\Lambda$ in place of $\Lambda$.
\end{lemma}
\begin{proof} 
In order to show that the map is surjective, we pull an $i$-cycle $\beta$ that hits the boundary of $|\P_\tau|$ into the interior. 
The cycle $\beta$ is a finite sum of simplices $\tau$ each carrying a section $a_\tau$ of $\iota_*\bigwedge^p\Lambda$. The section $a_\tau$ is defined on some neighbourhood $U_\tau$ of $\tau$. The union of the sets $U_\tau$ forms an open cover of $\beta$. Since $|\P_\tau|$ strongly retracts to a point in its interior, we may move $\beta$ along such a retraction. For a sufficiently small movement towards the retraction, we have that each cell $\tau$ of $\beta$ is still contained in $U_\tau$ but doesn't meet the boundary of $|\P_\tau|$ anymore. The resulting cycle $\beta'$ differs from the original one by the boundary of a chain which is obtained as the trace of the retraction from $\beta$ to $\beta'$. In a similar fashion, also relations between cycles can be homotoped into $W_\tau$ which proves the map in the assertion is an isomorphism.
\end{proof}

\begin{proposition} 
\label{hypothesis-for-lemma}
Let $B$ be a symple integral affine $n$-manifold with singularities $\Delta$ together with a triangulation $\P$ so that $\Delta$ is a simplicial subcomplex.
We assume the triangulation is fine enough so that each closed star $|\P_\tau|$ is contractible and contained in a symple model.
For $\tau\in\P$, let $W_\tau:=\operatorname{int}(|\P_\tau|)$ denote the the open star of $\tau$ then there is a commutative diagram 
\begin{equation}
\label{eq-key-map-in-prop}
\begin{gathered}
\xymatrix@C=30pt
{ 
H_0(|\P_\tau|,\iota_*\bigwedge^p\Lambda)\ar[r]\ar[d] & \Hom(\Gamma(W_\tau,\iota_*\bigwedge^p\check\Lambda),\ZZ)\ar[d]\\
H_0(|\P_\tau|,\iota_*\bigwedge^p\Lambda_\QQ)\ar[r]  & \Hom(\Gamma(W_\tau,\iota_*\bigwedge^p\check\Lambda_\QQ),\QQ)
}
\end{gathered}
\end{equation}
and $H_1(|\P_\tau|,\iota_*\Lambda_\QQ)=0.$
Similar statements hold after swapping $\Lambda$ and $\check\Lambda$ in the assertions.
\end{proposition}

If $(B,\P)$ is a symple integral affine manifold with polyhedral decomposition $\P$ and straightened discriminant in the sense of \cite{logmirror1} then $(\P^\bary)^\bary$ gives a triangulation satisfying the assumptions of Proposition~\ref{hypothesis-for-lemma} if $B$ is compact. For non-compact $B$ this goes similarly noting that some infinite (but locally finite) triangulation is needed for the unbounded cells in $B$.

\begin{proof}[Proof of Proposition~\ref{hypothesis-for-lemma}] %[Proof of Proposition~\ref{hypothesis-for-lemma}] 
The existence of the diagram is provided by Lemma~\ref{lemma-sections} for $n>1$ and the case $n=1$ ($\Delta=\emptyset$) follows from Example~\ref{ex-pair-local-sys}.

We prove the statement that $H_1(|\P_\tau|,\iota_*\Lambda_\QQ)=0$. The case $n=1$ follows because $|\P_\tau|$ is contractible and $\iota_*\Lambda_\QQ$ constant, so assume $n\ge 2$. Note that $\partial|\P_\tau|$ is homeomorphic to a sphere also in the case where $\tau\subset\partial B$.
By Lemma~\ref{lemma-homotope-inside}, it suffices to prove $H_1(U,\iota_*\Lambda_\QQ)=0$ for $U$ a closed subset of $W_\tau$ that is obtained by removing an $\epsilon$-neighborhood of $\partial|\P_\tau|$ and with the property that it contains a set of generators of $H_1(W_\tau,\iota_*\Lambda_\QQ)$.
We may assume that $U$ is also triangulated (by a new triangulation $\shP$).
We are going to apply the induction assumption to show that $H_1(\partial U,\iota_*\Lambda_\QQ)=0$.
We view $\iota_*\Lambda_\QQ$ as the functor on $\shP$ given by
$\sigma\mapsto \Gamma(\sigma,\iota_*\Lambda_\QQ)$. Let $\partial\shP$ denote the triangulation of $\partial U$ induced from $\shP$.
By \eqref{eq-iso-homology-chain-complex}, we have $H_1(\partial U,\iota_*\Lambda_\QQ)=H_1(\partial\shP,\iota_*\Lambda_\QQ)$.
Let $\omega\in\partial\shP$ and $C$ be the closed star of $\omega$ in $\partial\shP$. We want to show that $H_1(C,\iota_*\Lambda_\QQ)=0$.
Let $X$ be the symple model that $|\P_\tau|$ is contained in by assumption. 
By Remark~\ref{rem-dir-sum}, we may assume $r=1$, i.e.~$\Delta\cap X$ is given by $Y\times\check Y$.
We claim that $X$ splits of a trivial factor on $C$ in the sense of the localization in Remark~\ref{wlog-p-origin}. 
Indeed, the transition from $|\P_\tau|$ to $U$ ensures that $\partial U$ misses the zero-stratum of $Y\times\check Y$. 
It also ensures that there is a unique minimal stratum of $Y\times\check Y$ met by $C$ which is at least one-dimensional (unless $C\cap\Delta=\emptyset$ in which case $H_1(C,\iota_*\Lambda_\QQ)=0$ is obvious). We therefore find a monodromy invariant vector $v$ so that the symple model $X$ splits as 
$X\cong \RR v\times X'$ in a neighborhood of $C$ in the sense of Remark~\ref{wlog-p-origin}.
We obtain an induced splitting 
$$\iota_*\Lambda_\QQ\cong \iota'_*\Lambda'_\QQ \oplus\QQ v$$ 
and for $\iota'_*\Lambda'_\QQ$ we obtain from the induction assumption that $H_1(C,\iota'_*\Lambda'_\QQ)=0$ and thus $H_1(C,\iota_*\Lambda_\QQ)=0$ for all closed stars $C$ in $\partial\shP$.
Plugging this into Lemma~\ref{iso-via-clstars} yields an isomorphism
\begin{equation}
\label{induction-iso}
H_k(\partial U,\iota_*\Lambda_\QQ) \ra \Hom(H^k(\partial U,\iota_*\check\Lambda),\QQ)
\end{equation}
for $k=0,1$.
Next consider the isomorphism of long exact sequences that we have by Theorem~\ref{cor-iso-ho-coho},
\begin{equation*}
\resizebox{\textwidth}{!}{
\xymatrix@C=30pt
{ 
\dots\ar[r]& H_k(\partial U) \ar[r]\ar^\sim[d] & H_k( U) \ar[r]\ar^\sim[d] & H_k(U,\partial U) \ar[r]\ar^\sim[d]& H_{k-1}(\partial U) \ar[r]\ar^\sim[d] &\dots\\
\dots\ar[r]& H^{n-(k+1)}(\partial U) \ar[r] & H^{n-k}( U,\partial U) \ar[r] & H^{n-k}(U) \ar[r] &H^{n-k}(\partial U) \ar[r] &\dots
}
}
\end{equation*}
where all coefficients are $\iota_*\Lambda_\QQ$, also in what follows next. 
Lemma~\ref{lemma-local-sections} says $H^{n-k}(U) = 0$ for $k\neq n$ 
and therefore
\begin{equation}
\label{iso-in-front}
H_k(\partial U)= H_k(U)
\end{equation}
holds for $k<n-1$. 
For $n>1$, Lemma~\ref{lemma-sections} says the restriction 
\begin{equation}
\label{iso-in-back}
H^0(U)\ra H^0(\partial U)
\end{equation}
is an isomorphism.
In order to show that $H_1(U)=0$ for $n\ge 3$, 
we insert \eqref{iso-in-front} and Theorem~\ref{H1-of-punctured-is-zero} into \eqref{induction-iso}.
For showing that $H_0(U) \ra \Hom(H^0(U),\QQ)$ is an isomorphism, we set $k=0$ and plug \eqref{iso-in-front} and \eqref{iso-in-back} into \eqref{induction-iso},
concluding the induction step for $n\ge 3$.

The case $n=2$ can be treated as follows; assume $n=2$. 
The case $k=0$ goes the same way as before when $n>2$. 
However, $\rk H^1(\partial U)\ge 1$ so we need a different argument for $k=1$. Regarding the isomorphism of exact sequences
\begin{equation*}
\resizebox{0.9\textwidth}{!}{
\xymatrix@C=30pt
{ 
\dots\ar[r] &H_{2}(U,\partial U) \ar[r]\ar^\sim[d] & H_1(\partial U) \ar[r]\ar^\sim[d] & H_1( U) \ar[r]\ar^\sim[d] & H_1(U,\partial U) \ar[r]\ar^\sim[d]& H_{0}(\partial U) \ar[r]\ar^\sim[d] &\dots\\
\dots\ar[r] &H^{0}(U) \ar^\sim[r] & H^{0}(\partial U) \ar[r] & H^{1}( U,\partial U) \ar[r] & \protect\underbrace{H^{1}(U)}_{=0} \ar[r] &H^{1}(\partial U) \ar[r] &\dots,
}
}
\end{equation*}
since the restriction $H^{0}(U)\ra H^{0}(\partial U)$ is an isomorphism by Lemma~\ref{lemma-sections}, the diagram gives $H_1(U)=0$ and we are done also with $n=2$.
\end{proof} 

\begin{proof}[Proof of Theorem~\ref{pairingthm}] 
The functor $\tau\mapsto \Hom(\Gamma(W_\tau,\iota_*\bigwedge^p\check\Lambda),\ZZ)$ is contra-variant and
the $k$th homology of the associated chain complex comes with a natural map to $\Hom(H^k(B,\iota_*\bigwedge^p\check\Lambda),\ZZ)$. This goes similarly over $\QQ$.
The horizontal maps in \eqref{eq-key-map} give maps of functors so that Lemma~\ref{iso-via-clstars} produces the top horizontal map of the diagram in the assertion both over $\ZZ$ and $\QQ$.

Let $T\supset \Delta$ be an open tubular neighborhood of $\Delta$ in $B$ so that $B\setminus T$ permits a locally finite triangulation (i.e.~form the neighborhood in the PL category). We use Lemma~\ref{lem-tube-is-mfd} to know that $B\setminus T$ is a topological manifold with boundary.
We claim that the natural map $H_q(B\setminus\Delta,\bigwedge^p\Lambda)\ra \Hom(H^q(B\setminus\Delta,\bigwedge^p\check\Lambda),\ZZ)$ is isomorphic to 
$H_q(B\setminus T,\bigwedge^p\Lambda)\ra \Hom(H^q(B\setminus T,\bigwedge^p\check\Lambda),\ZZ)$.
The isomorphism $H^q(B\setminus T,\bigwedge^p\check\Lambda)=H^q(B\setminus\Delta,\bigwedge^p\Lambda)$ follows by a \v{C}ech cover argument. 
The similar identification for the homology groups follows from the observation that we can push cycles as well as relations of cycles from $B\setminus\Delta$ into $B\setminus T$ similar to what we did in the proof of Lemma~\ref{lemma-homotope-inside}.
The map
$H_q(B\setminus T,\Lambda_\QQ)\ra \Hom(H^q(B\setminus T,\Lambda_\QQ),\QQ)$ agrees with the map $\alpha$ in Example~\ref{ex-pair-local-sys} so 
that the existence of the commutative diagram in the assertion of Theorem~\ref{pairingthm} follows from Remark~\ref{remark-compatibility-subcomplex}. This also proves that the bottom horizontal map is an isomorphism over $\QQ$.

To finally prove that the top horizontal map is an isomorphism over $\QQ$ when $p=1$ and $q\le 1$, we combine Proposition~\ref{hypothesis-for-lemma} with Lemma~\ref{iso-via-clstars}.
\end{proof} 

%===========================================================
%
%		section: Proofs for the intersection theoretic theorems
%
%===========================================================

\section{Proofs for the intersection pairing theorems}
Let $B$ be an integral affine manifold with singularities $\Delta$ and $\iota\colon B\setminus \Delta \ra B$ the inclusion of the regular locus.
There is a map of constructible sheaves
$$\iota_*\bigwedge^{p_1}\Lambda\ \otimes\ \iota_*\bigwedge^{p_2}\Lambda\ \ra\ \iota_*\bigwedge^{p_1+p_2}\Lambda,$$
indeed by the characterization of sections of $\iota_*\bigwedge^p \Lambda$ via monodromy invariant sections of a nearby stalk in Lemma~\ref{lemma-sections}, the map is well-defined. 
If $B$ is $n$-dimensional and orientable then $\iota_*\bigwedge^{n}\Lambda\cong \ZZ$ and we may pick a generator $\Omega\in \Gamma(B,\iota_*\bigwedge^{n}\Lambda)$ by means of which we obtain a pairing
\begin{equation}
\label{eq-sheaf-pairing-perfect}
\iota_*\bigwedge^{p}\Lambda\ \otimes\ \iota_*\bigwedge^{n-p}\Lambda\ \ra\ \ZZ.
\end{equation}
\begin{lemma}
\label{lemma-sheaf-pairing-perfect}
The pairing \ref{eq-sheaf-pairing-perfect} is perfect, i.e.~it induces an isomorphism of sheaves
$\iota_*\bigwedge^{p}\Lambda\stackrel{\Omega}\lra \Hom(\iota_*\bigwedge^{n-p}\Lambda,\ZZ)=\iota_*\bigwedge^{n-p}\check\Lambda$.
\end{lemma}
\begin{proof} 
The sheaf pairing is clearly perfect on $B\setminus \Delta$, so we obtain an isomorphism  
$\bigwedge^{p}\Lambda\ra \Hom(\bigwedge^{n-p}\Lambda,\ZZ)$. Note that $\Hom(\bigwedge^{n-p}\Lambda,\ZZ)=\bigwedge^{n-p}\check\Lambda$.
Taking $\iota_*$ and using $\iota_*\Hom(\bigwedge^{n-p}\Lambda,\ZZ)=\Hom(\iota_*\bigwedge^{n-p}\Lambda,\ZZ)$ gives the claim.
\end{proof}

\begin{theorem}
\label{thm-pairings-agree}
Assume $B$ is compact and oriented by $\Omega\in\Gamma(B,\iota_*\bigwedge^n\Lambda)$. Set $B^\circ=B\setminus\partial B$ and let $H^\bullet_c$ denote cohomology with compact support. For each $p,q$, there is a commutative diagram
\begin{equation}
\label{eq-pairings-agree}
\begin{gathered}
\xymatrix@C=30pt
{ 
H_q(B,\iota_*\bigwedge^p\Lambda) \ar[r]\ar^{\Omega}[d] & \Hom(H^q(B,\iota_*\bigwedge^p\check\Lambda),\ZZ)\ar^{\Omega}[d] \\
H^{n-q}_c(B^\circ,\iota_*\bigwedge^p\Lambda) \ar[r]  & \Hom(H^{q}(B,\iota_*\bigwedge^{n-p}\Lambda),\ZZ) 
}
\end{gathered}
\end{equation}
with left vertical map the isomorphism given by Theorem~\ref{cor-iso-ho-coho} using $H^{i}_c(B^\circ,\iota_*\bigwedge^p\Lambda)=H^{i}(B,\partial B;\iota_*\bigwedge^p\Lambda)$, bottom horizontal map obtained from the cup product via \eqref{eq-sheaf-pairing-perfect}, the right vertical map is the isomorphism given by contracting $\Omega$ and the top horizontal map is the top horizontal map of Theorem~\ref{pairingthm}.
\end{theorem}
\begin{remark}
A diagram similar to \eqref{eq-pairings-agree} exists also over $\QQ$ and if $B$ is symple, then we know for its top horizontal arrow to be an isomorphism for $p=q=1$ by Theorem~\ref{pairingthm}. Thus, the bottom horizontal map is also an isomorphism in this case. In particular, we conclude an equality of affine Hodge numbers $h^{n-1,1}=h^{1,n-1}$ entirely from the study of the affine geometry. That such an equality is expected to hold also for other degrees gives yet further evidence that the more general pairing \eqref{eq-gen-pairing} ought to be perfect in the symple case. 
On the other hand, it is known that $\dim\Gr_F^{1}H^3(X,\CC)\neq \dim\Gr_F^{2}H^3(X,\CC)$ holds for a conifold Calabi--Yau 3-fold $X$ where $F^\bullet$ is the Hodge filtration. We obtain a new perspective on Counter-Example~\ref{counterexample-conifold} if we assume that \eqref{affine-to-Hodge} is an isomorphism for a conifold (which is currently unknown).
\end{remark}
\begin{proof}[Proof of Theorem~\ref{thm-pairings-agree}] 
We pick a triangulation $\P$ of $B$ so that $\Delta$ is a simplicial subcomplex, that
$\iota_*\bigwedge^p \Lambda$ is $\P$-acyclic, each closed star $|\P_\tau|$ is contractible and $H^i(|\P_\tau|,\iota_*\bigwedge^p \Lambda)=0$ for $i>0$.
For $U=|\P_\tau|$ a closed star and $U^\circ$ its interior, consider the following diagram
\begin{equation}
\label{first-show-this-commutes}
\begin{gathered}
\xymatrix@C=30pt
{ 
H_0(U,\iota_*\bigwedge^p\Lambda) \ar[r]\ar^{\Omega}[d] & \Hom(\Gamma(U,\iota_*\bigwedge^p\check\Lambda),\ZZ)\ar^{\Omega}[d] \\
H^{n}_c(U^\circ,\iota_*\bigwedge^p\Lambda) \ar@{->}[r]  & \Hom(\Gamma(U,\iota_*\bigwedge^{n-p}\Lambda),\ZZ) 
}
\end{gathered}
\end{equation}
where the top horizontal map is the top horizontal map in \eqref{eq-key-map}, $H^\bullet_c$ denotes cohomology with compact support and in fact $H^{i}_c(U^\circ,\iota_*\bigwedge^p\Lambda)=H^{i}(U,\partial U;\iota_*\bigwedge^p\Lambda)$. 
Thus, for the left vertical map, we may take the suitable vertical isomorphism of Theorem~\ref{cor-iso-ho-coho}.
The bottom right term is $\Hom(\Gamma(U,\iota_*\bigwedge^{n-p}\Lambda),\ZZ)=\Gamma(U,\iota_*\bigwedge^{n-p}\check\Lambda)=\big(\iota_*\bigwedge^{n-p}\check\Lambda\big)_y$
for $y\in \tau$ a suitable point. Similarly for the top right one,
$\Hom(\Gamma(U,\iota_*\bigwedge^p\check\Lambda),\ZZ)=\Gamma(U,\iota_*\bigwedge^p\Lambda)=\big(\iota_*\bigwedge^p\Lambda\big)_y$, so we may take 
the isomorphism from Lemma~\ref{lemma-sheaf-pairing-perfect} for the right vertical map.

Now there is a unique map to use for the bottom horizontal arrow to make the diagram commute. 
We claim it is the one given by the cup product pairing 
\begin{equation}
\label{eq-bottom-pairing-map}
H^{n}_c(U^\circ,\iota_*\bigwedge^p\Lambda)\otimes H^0(U^\circ,\iota_*\bigwedge^{n-p}\Lambda)\ra H^{n}_c(U^\circ,\iota_*\bigwedge^{n}\Lambda)\underset{\Omega}{\cong}\ZZ
\end{equation}
that uses \eqref{eq-sheaf-pairing-perfect} and for the last isomorphism uses that $U^\circ$ is oriented and that $\Omega$ trivializes $\iota_*\bigwedge^{n}\Lambda$.
The map \ref{eq-bottom-pairing-map} is computed by taking an $n$-\v{C}ech-co-cycle $\alpha$ for a suitable cover of $(U,\partial U)$, applying the a global section
$\beta$ of $\iota_*\bigwedge^{n-p}\Lambda$ to it under \eqref{eq-sheaf-pairing-perfect} and then comparing the resulting co-cycle $\alpha\wedge\beta$ modulo co-boundaries with $\Omega\xi$ where $\xi\in H^{n}_c(U,\ZZ)$ stands for a representative of the class chosen by the orientation of $B$.
We claim that, under the two vertical isomorphisms, the corresponding procedure computes the top horizontal map. This will be slightly easier to see by passing (using the long exact sequence of the pair $(U,\partial U)$) through another isomorphic version of \ref{eq-bottom-pairing-map}, namely
\begin{equation}
\label{eq-bottom-pairing-map2}
H^{n-1}(\partial U,\iota_*\bigwedge^p\Lambda)\otimes H^0(\partial U,\iota_*\bigwedge^{n-p}\Lambda)\ra H^{n-1}(\partial U,\iota_*\bigwedge^{n}\Lambda)\cong\ZZ.
\end{equation}
Here we use $n\ge 2$ to have $H^i(U,\shF)=0$ for $i=n,n-1$ and $\shF=\iota_*\bigwedge^p\Lambda,\iota_*\bigwedge^{n}\Lambda$ (by Lemma~\ref{lemma-local-sections}) which we may do because the case $n=1$ is straightforward.
The same computation with $\alpha,\beta$ as before holds for \eqref{eq-bottom-pairing-map2}. Under the Poincar\'e--Lefschetz isomorphism (Theorem~\ref{cor-iso-ho-coho}) on the outer terms in  \eqref{eq-bottom-pairing-map2}, the map \eqref{eq-bottom-pairing-map2} becomes 
\begin{equation}
\label{eq-bottom-pairing-map3}
H_0(\partial U,\iota_*\bigwedge^p\Lambda)\otimes \Gamma(\partial U,\iota_*\bigwedge^{n-p}\Lambda)\ra H_0(\partial U,\iota_*\bigwedge^{n}\Lambda)=\ZZ\Omega.
\end{equation}
which is taking a representative of a homology class $\alpha\in H_0(\partial U,\iota_*\bigwedge^p\Lambda)$, applying 
$\beta\in \Gamma(\partial U,\iota_*\bigwedge^{n-p}\Lambda)$ to it and comparing the result with $\Omega$ supported on a point. 
After plugging $\iota_*\bigwedge^{n-p}\Lambda=\iota_*\bigwedge^{p}\check\Lambda$ into the middle term in \eqref{eq-bottom-pairing-map3}, 
one checks now that this agrees with how the top horizontal map in \eqref{eq-key-map} is constructed, so we do find that \eqref{first-show-this-commutes} commutes with the bottom map being the natural one.

For the remainder of the proof, we want to conclude the assertion from the commutativity of \eqref{first-show-this-commutes}. Consider the double complex
\begin{equation}
\label{eq-double-complex-cohomology}
\begin{gathered}
\resizebox{.9\textwidth}{!}{
\xymatrix@C=30pt
{ 
\vdots\ar[d]&\vdots\ar[d]\\
\bigoplus_{\tau\in\P^{[0]}} H^{n-2}(\Gr_F^{n-2}C^\bullet(\P_\tau,\partial\P_\tau))\ar[d] & \bigoplus_{\tau\in\P^{[1]}} H^{n-2}(\Gr_F^{n-2}(C^\bullet(\P_\tau,\partial\P_\tau)) \ar[l]\ar[d] & \ar[l]\dots\\
\bigoplus_{\tau\in\P^{[0]}} H^{n-1}(\Gr_F^{n-2}C^\bullet(\P_\tau,\partial\P_\tau))\ar[d] & \bigoplus_{\tau\in\P^{[1]}} H^{n-1}(\Gr_F^{n-2}(C^\bullet(\P_\tau,\partial\P_\tau)) \ar[l]\ar[d] & \ar[l]\dots\\
\bigoplus_{\tau\in\P^{[0]}} H^{n}(\Gr_F^{n-2}C^\bullet(\P_\tau,\partial\P_\tau)) & \bigoplus_{\tau\in\P^{[1]}} H^{n}(\Gr_F^{n-2}(C^\bullet(\P_\tau,\partial\P_\tau)) \ar[l] & \ar[l]\dots
}
}
\end{gathered}
\end{equation}
where $C^\bullet(\P_\tau,\partial\P_\tau)$ refers to the complex defined in \eqref{def-coho-cone} for the sheaf $\iota_*\bigwedge^p\Lambda$ and the cover $\{W_\sigma|\sigma\in\P^{[n]}_\tau\}$ and $F$ is the filtration defined before Lemma~\ref{lemma-gradedcoho}. The vertical maps are the $d_1$-differentials introduced before Proposition~\ref{prop-map-ho-coho}. The horizontal maps are the obvious ones using Lemma~\ref{lemma-gradedcoho}.
Let $\check T^\bullet$ denote the total complex of \eqref{eq-double-complex-cohomology}. We claim to have a commutative diagram of complexes (with $i$ running)
\begin{equation}
\label{eq-final-diagram-for-comp}
\begin{gathered}
\xymatrix@C=30pt
{ 
C_i(B,\iota_*\bigwedge^p\Lambda) \ar^{f}[d] & \ar[l] T^i\ar[d]\ar[r]  & \Hom(C^i(\{W_v\},\iota_*\bigwedge^p\check\Lambda),\ZZ)\ar[d]\\
H^{n-i}\Gr_F^{n-i}C^{\bullet}(\{W_\sigma\},\iota_*\bigwedge^p\Lambda) & \ar[l] \check T^i\ar[r] & \Hom(C^i(\{W_v\},\iota_*\bigwedge^{n-p}\Lambda),\ZZ)
}
\end{gathered}
\end{equation}
where the top row was given in the proof of Lemma~\ref{iso-via-clstars}.
The double complex \eqref{eq-double-complex-cohomology} is isomorphic to the one given in \eqref{eq-double-complex-homology} on the nose which
gives us the middle vertical isomorphism $T^\bullet\ra\check T^\bullet$. 
Furthermore, the rows of \eqref{eq-double-complex-cohomology} have homology concentrated in the first column. The resulting homology in the first column becomes the complex given in the bottom left of \eqref{eq-final-diagram-for-comp} which explains the bottom left horizontal map in \eqref{eq-final-diagram-for-comp}. In fact, we conclude the entire left square of \eqref{eq-final-diagram-for-comp} as the isomorphism $T^\bullet\ra\check T^\bullet$ induces the map $f$ given in Proposition~\ref{prop-map-ho-coho}.
The homology of the $i$th column of \eqref{eq-double-complex-cohomology} computes $\bigoplus_{\tau\in\P^{[i]}} H^j_c(W_\tau,\iota_*\bigwedge^p\Lambda)$ with $H^n_c(W_\tau,\iota_*\bigwedge^p\Lambda)$ sitting at the bottom. The commuting of the right square in \eqref{eq-final-diagram-for-comp} now follows from plugging the commutative diagram \eqref{first-show-this-commutes} into Lemma~\ref{iso-via-clstars} which makes the horizontal maps in the right square of \eqref{eq-final-diagram-for-comp}.
Taking (co)homology of \eqref{eq-final-diagram-for-comp} yields the commutative diagram of the assertion.
\end{proof}

\begin{corollary}
\label{pairing-is-cap-product}
The pairing $H_{p,q}\otimes H^{p,q}\ra\ZZ$ given in Theorem~\ref{pairingthm} is the composition of the cap product 
$H_q(B,\iota_*\bigwedge^p\Lambda)\otimes H^q(B,\iota_*\bigwedge^p\check\Lambda)\ra H_0(B,\iota_*\bigwedge^p\Lambda\otimes\iota_*\bigwedge^p\check\Lambda)$
with the map $\op{tr}$ given in Lemma~\ref{lemma-cap-zero}.
\end{corollary}
\begin{proof}
By Corollary~\ref{PLmap-is-the-same}, the left vertical map in Theorem~\ref{thm-pairings-agree} agrees with the Poincar\'e--Lefschetz map in \cite{Br97}.
The bottom horizontal map in Theorem~\ref{pairingthm} is the cup product by the statement of the theorem. 
Since cup and cap product are compatible under Poincar\'e--Lefschetz by \cite[Theorem 10.1]{Br97}, in view of Lemma~\ref{lemma-cap-zero}, we conclude that the top horizontal map in Theorem~\ref{thm-pairings-agree} is given by the cap product.
\end{proof}

\begin{proof}[Proof of Theorem~\ref{thm-intersect}] 
That (1) and (2) coincide is Theorem~\ref{thm-pairings-agree}.
We use Corollary~\ref{pairing-is-cap-product} and then \cite[V-\S11]{Br97} and localize the intersection product to a small neighborhood of each intersection point where we may assume by the hypothesis on $V$ and $W$ that $V$ and $W$ are manifolds.
That (1) and (2) agree with (3) now follows from the geometric interpretation of the intersection product as given in \cite[Theorem VI.11.9,\,p.\,372]{Br93} enhanced by permitting coefficients in locally constant sheaves as in \cite[V-\S11]{Br97}.
\end{proof}

\begin{proof}[Proof of Theorem~\ref{thm-zharkov}] 
Let $x\in V\cap W$ be an intersection point and $U$ a small neighborhood of $x$ where $V$ and $W$ are manifolds and we can trivialize $\Lambda|_U$ and $\check\Lambda|_U$. Since $\check\Lambda|_U$ frames $\shT^*_U$, we don't need to distinguish between the stalk and global sections and may use the notation $\Lambda,\check\Lambda$ to refer to both. Let $\xi_V\in\bigwedge^p\Lambda$ and $\xi_W\in \bigwedge^{n-p}\Lambda$ denote the section carried by $V$ and $W$ respectively.
If $\xi_V$ decomposes into a sum of $r$ simple primitive forms, then we replace $V$ by $r$ copies of $V$, each carrying one simple primitive summand. The same holds for $\xi_W$. We therefore assume from now on that $\xi_V$ and $\xi_W$ are simple and primitive, say 
$\xi_V=v_1\wedge...\wedge v_p$ where $v_i\in\Lambda$.

We have $\pi^{-1}(U)=T^*_{U}/\check\Lambda$ and there is a natural translation action of $T^*_{U}$ on the fibers of $\pi|_{\pi^{-1}(U)}$.
Let $\xi_V^\perp$ refer to the subgroup of $T^*_{U}$ given by $v_1^\perp\cap...\cap v_p^\perp$.
The construction of the ordinary homology cycles $\beta_V,\beta_W$ makes use of a section $S\colon U\ra X$. The cycle $\beta_V$ over $U$ is then given by 
$\beta_V:=\xi_V^\perp.S(V)$ where $\xi_V^\perp.S(V)$ refers to the orbit of $S(V)$ under the subgroup $\xi_V^\perp$. A similar construction yields $\beta_W$. 

The orientation of $\beta_V$ is obtained by giving $S(V)$ the orientation of $V$ and orienting $\xi_V^\perp$ so that contracting the dual $\check\Omega$ of $\Omega$ by $\xi_V$ gives the orientation form. The orientation for $\beta_W$ is given similarly.
The orbits of $\xi_V^\perp$ and $\xi_W^\perp$ intersect in $|\frac{\xi_V\wedge \xi_W}\Omega|$ many points and each comes with the same intersection multiplicity computed as follows.

The tangent space of $\beta_V$ at the intersection point $S(x)$ is $T_{\beta_V,S(x)}=T_{V,x}\oplus \xi_V^\perp$ and similarly $T_{\beta_W,S(x)}=T_{W,x}\oplus \xi_W^\perp$ is the tangent space at $\beta_W$.
Note that $\dim\xi_V^\perp=n-p$ and $\dim T_{W,x}={n-q}$. So the orientation of $T_{\beta_V,S(x)}\oplus T_{\beta_W,S(x)}$ is $(-1)^{(n-p)(n-q)}$ times the orientation of $T_{V,x}\oplus T_{W,x}\oplus \xi_V^\perp\oplus \xi_W^\perp$. The latter is $\eps(T_{V,x},T_{W,x})$ multiplied with the sign of 
$\frac{\iota_{\xi_V}\check\Omega\wedge\iota_{\xi_W}\check\Omega}{\check\Omega}$ which in turn is $(-1)^{(n-p)p}$ times the sign of $\iota_{\xi_V\wedge \xi_W}\check\Omega=\frac{\xi_V\wedge \xi_W}{\Omega}$.
Checking that $(-1)^{(n-p)(n-q)+(n-p)p} = (-1)^{(n-p)(q-1)}$ finishes the proof.
\end{proof}

\appendix

%===========================================================
%
%		appendix A: Identification of simplicial and singular homology
%
%===========================================================

\section{Identification of simplicial and singular homology}
\label{general-constructible-sheaves}
We could not find the following results on constructible sheaves in
the literature, so we provide proofs here. 
Recall from
\cite[\S2.1]{Ha02} that a $\triang$-{\bf complex} is a CW-complex
where each closed cell comes with a distinguished surjection to it
from the (oriented) standard simplex with compatibility between
sub-cells and faces of the simplex. We denote the relative interior of a simplex $\tau$ by $\tau^\circ$.
For a $\triang$-complex $X$, we use the notation $X=\coprod_{\tau\in T} \tau^\circ$ where $T$ is a
set of simplices for each of which we have the {\bf characteristic
map} $j_\tau\colon \tau\ra X$  that restricts to a homeomorphism on $\tau^\circ$. Let $X=\coprod_{\tau\in
T}\tau^\circ$ be a $\triang$-complex. 
\begin{definition} 
\begin{enumerate}
\item A sheaf of Abelian groups $\shF$ on $X$
is $T$-{\bf constructible} if $\shF|_{\tau^\circ}$ is a constant
sheaf for each $\tau\in T$. 
\item A $T$-constructible sheaf $\shF$ is $T$-{\bf acyclic} if
$$H^i(\tau,j_\tau^{-1}\shF)=0$$
holds for each $\tau\in T$ and $i>0$.
\end{enumerate}
\label{def-constructible}
\end{definition}
We denote by $T^\bary$ the barycentric subdivision of $T$.
\begin{lemma} 
\label{baryacyclic}
Let $\shF$ be a $T$-constructible sheaf then $\shF$ is $T^\bary$-acyclic.
\end{lemma}
\begin{proof}
This goes by a retraction-to-the-stalk argument as in \cite[Proof of Lemma~5.5]{logmirror1}, here are the details:
let $\sigma\in T^\bary$. We write $\shF$ for $j_\sigma^{-1}\shF$.
Let $v_{\tau_0},...,v_{\tau_r}$ be the vertices of $\sigma$ so that $v_{\tau_i}$ is the barycenter of $\tau_i\in T$ and $\tau_0\subset\tau_1\subset...\subset\tau_r$.
Make $v_{\tau_0}$ be the origin then for every open neighborhood $U$ of $v_0$ in $\sigma$ and every $0<\lambda<1$ we have $\Gamma(U,\shF)=\Gamma(\lambda U,\shF)$ since $\shF$ is T-constructible.
This implies that $H^i(\sigma,\shF)=(R^i\id_*\shF)_{v_0}$ but for every sheaf $\shF$ and $i>0$ holds $R^i\id_*\shF=0$, so we are done.
\end{proof}

Let $S\subseteq T$ so that $A=\coprod_{\tau\in S}\tau^\circ$ is a (closed) subcomplex of $X$. Let $\shF$ be a $T$-constructible sheaf on $X$.
\begin{definition}
\label{def-homology}
\begin{enumerate}
\item
We denote by $H_i^T(X,A;\shF)$ the relative simplicial homology with
coefficients $\shF$, i.e.~it is computed by the differential graded $\ZZ$-module
\[
\bigoplus_{i\ge 0} C^T_i(X,A;\shF)\qquad \hbox{ where }\qquad
C^T_i(X,A;\shF):=\bigoplus_{{\tau\in T\setminus
S}\atop{\dim\tau=i}}\Gamma(\tau,j_\tau^{-1}\shF)
\]
with the usual differential $\partial\colon C^T_i(X,A;\shF)\ra
C^T_{i-1}(X,A;\shF)$ whose restriction/projection to
$\Gamma(\sigma,j_\sigma^{-1}\shF)\ra \Gamma(\tau,j_\sigma^{-1}\shF)$ for
$\tau\subset\sigma$ a facet inclusion is given by the restriction
map multiplied by
\[
\eps_{\tau\subset\sigma}
=\left\{\begin{array}{ll}
+1, & n_\tau\wedge\ori_{\tau}=+\ori_\sigma \\
-1, & n_\tau\wedge\ori_{\tau}=-\ori_\sigma
\end{array}\right.
\]
where $\ori_{\tau},\ori_{\sigma}$ denote the orientation of
$\tau,\sigma$ respectively and $n_\tau$ is the outward normal of
$\sigma$ along $\tau$.

\item
On the other hand, one defines $H_i(X,A;\shF)$, the singular
homology with coefficients $\shF$, in the usual way (see for
example \cite[VI-\S12]{Br97}) where chains $C_i(X,A;\shF)$ are finite formal
sums over singular $i$-simplices in $X$ modulo singular
$i$-simplices in $A$.
\end{enumerate}
We denote by $\overrightarrow{C}_i(X,A;\shF)$ the direct limit of
$C_i(X,A;\shF)$ under the barycentric subdivision operator on
singular chains, see \cite[V-\S1.3]{Br97}. 
When $X,A,\shF$ are unambiguous, 
we write $C^T_i$ or $C^T_i(\shF)$ for $C^T_i(X,A;\shF)$; we write $C_i$ or $C_i(\shF)$ for $C_i(X,A;\shF)$ and we write $\overrightarrow{C}_i$ or $\overrightarrow{C}_i(\shF)$ for
$\overrightarrow{C}_i(X,A;\shF)$.
\end{definition}

Note that there is a natural map of complexes $C^T\ra C^{T^\bary}$ that maps the section of $\shF$ on a simplex $\tau\in T$ to its restrictions to all $\tau'\in T^\bary$ that have the property $\tau'\subset\tau$ and $\dim\tau=\dim\tau'$.

\begin{lemma} 
\label{reduce-to-barycentric}
Let $X=\coprod_{\tau\in T} \tau^\circ$ be a $\triang$-complex, $A$ a
subcomplex and $\shF$ be a $T$-acyclic sheaf on $X$. 
For every $i$, the restriction map $C^T\ra C^{T^\bary}$ induces an isomorphism
\[
H^{T}_i(X,A;\shF) \lra H^{T^\bary}_i(X,A;\shF).
\]
\end{lemma}

\begin{proof}
For $\tau\in T^\bary$, let $\hat\tau$ denote the smallest simplex in
$T$ containing $\tau$. Define
$$F_jC_i^{T^\bary}= \bigoplus_{{\tau\in T^\bary\setminus S^\bary}\atop{\dim\hat\tau\le j},\dim\tau=i} \Gamma(\tau,j_\tau^{-1}\shF).$$
We have that 
$\{0\}\subset F_0C_\bullet^{T^\bary}\subset F_1C_\bullet^{T^\bary} \subset ...$
defines an increasing filtration of $C_\bullet^{T^\bary}$ by sub-complexes. 
Set $\Gr_i^F C_\bullet^{T^\bary} = \frac{F_iC_\bullet^{T^\bary}}{F_{i-1}C_\bullet^{T^\bary}}$.
Note that 
$$\Gr^F_jC_i^{T^\bary}= \bigoplus_{{\tau\in T^\bary\setminus S^\bary}\atop{\dim\hat\tau=j},\dim\tau=i} \Gamma(\tau,j_\tau^{-1}\shF)$$
and $\Gr_i^F C_\bullet^{T^\bary}=0$ if $i>j$. 
By Lemma~\ref{baryacyclic}, $K^\bullet:=\Gr_i^F C_{i-\bullet}^{T^\bary}$ computes the cohomology of the pull-back of $\shF$ to $\coprod_{\hat\tau\in T\setminus S,\dim \hat\tau=i} \hat\tau$.
Since $\shF$ is $T$-acyclic, we conclude 
$$
H_k^\partial\Gr_i^F C_\bullet^{T^\bary} 
= \left\{\begin{array}{ll}
C_i^T=\displaystyle\oplus_{\hat\tau\in T\setminus S,\dim \hat\tau=i} \Gamma(\hat\tau,j_{\hat\tau}^{-1}\shF) & i=k\\
0,& i\neq k.
\end{array}\right.
$$
The spectral sequence of the filtered complex $(C_i^{T^\bary}, F)$ reads
$$E^{p,q}_1=H_{-p-q}\Gr^F_{-p}C_\bullet^{T^\bary} \Rightarrow H_{-p-q}C_\bullet^{T^\bary}.$$
By the vanishing that we just proved, we find that it is concentrated at $q=0$ and thus this spectral sequence degenerates at $E_2$ and the differential $d_1$ is just the usual differential $\partial$ on $C_\bullet^T$.
This means we have an isomorphism
$$ H_{-p} C^T_\bullet = \Gr^F_{-p}H_{-p}C^{T^\bary}_\bullet.$$
The filtration $F$ is concentrated in one degree on the cohomology group of $C^{T^\bary}_\bullet$, so we have
$\Gr^F_{-p}H_{-p}C^{T^\bary}_\bullet= H_{-p}C^{T^\bary}_\bullet$ which gives the claim.
\end{proof}

Let $j_k$ denote the inclusion of the complement of the
$(k-1)$-skeleton in $X$, i.e.~$$j_k\colon \left(\coprod_{{\tau\in
T}\atop{\dim\tau\ge k}}\tau^\circ \right)\hra X.$$ Consider the
decreasing filtration of $T$-constructible sheaves
\[
\shF=\shF^0\supset\shF^1\supset\shF^2\supset\ldots
\]
of $\shF$ defined by $\shF^k={(j_k)_!}{j_k^{-1}}\shF$. Let $i_k$ denote
the inclusion of the $k$-skeleton in $X$ and $i_{\tau^\circ}$ denote
the inclusion of $\tau^\circ$ in the $k$-skeleton. We have
\begin{equation} \label{eq-Gr-decompose}
\Gr^k_\shF = \shF^k/\shF^{k+1}=\bigoplus_{{\tau\in
T}\atop{\dim\tau=k}} (i_k)_*(i_{\tau^\circ})_!(\shF|_{\tau^\circ}).
\end{equation}

\begin{lemma} 
\label{reducetoGr}
We have the following commutative diagram with exact rows
\begin{equation*}
\xymatrix@C=30pt
{
0 \ar[r]& C_\bullet^{T^\bary}(\shF^{k+1})   \ar[r]\ar[d]&
C_\bullet^{T^\bary}(\shF^{k})               \ar[r]\ar[d]&
C_\bullet^{T^\bary}(\Gr^k_\shF)         \ar[r]\ar[d]& 0 \\
0 \ar[r]& \overrightarrow{C}_\bullet(\shF^{k+1}) \ar[r]&      
\overrightarrow{C}_\bullet(\shF^{k}) \ar[r]&      
\overrightarrow{C}_\bullet(\Gr^k_\shF) \ar[r]&       0. \\
}
\end{equation*}
A similar statement holds if we replace $0\ra\shF^{k+1}\ra \shF^{k}\ra\Gr^k_\shF\ra 0$ by any other short exact sequence of $T$-constructible sheaves.
\end{lemma}
\begin{proof}
The vertical maps are the universal maps, so commutativity is
straightforward. The left-exactness of the global section functor
leaves us with showing the exactness of the rows at the rightmost
non-trivial terms.  The first row is exact because for every $\tau\in T^\bary$ we have
$H^1(\tau,j_\tau^{-1}\shF^{k+1})=0$ by Lemma~\ref{baryacyclic}.
We show the surjectivity of 
$\overrightarrow{C}_\bullet(\shF^{k}) \ra
\overrightarrow{C}_\bullet(\Gr^k_\shF)$. Let $s\colon \tau\ra X$ be a
singular simplex and $g\in \Gamma(\tau,s^{-1}(\shF_k/\shF_{k+1}))$.
By the exactness of $s^{-1}$ and the surjectivity of
$\shF_k\ra\shF_k/\shF_{k+1}$, we find an open cover $\{U_\alpha\}$
of $\tau$ such that  $g|_{U_\alpha}$ lifts to
$\hat{g}_\alpha\in\Gamma(U_\alpha,\shF_k)$. By the compactness of
$\tau$, we may assume the cover to be finite. After finitely many
iterated barycentric subdivisions of $\tau$, we may assume each
simplex of the subdivision to be contained in a $U_\alpha$ for some
$\alpha$. Let $\tau'$ be such a simplex contained in $U_\alpha$,
then $g|_{\tau'}$ lifts to $\hat{g}_\alpha|_{\tau'}$ and we are done
since it suffices to show surjectivity after iterated barycentric
subdivision.
\end{proof}

\begin{theorem} 
\label{simplicial=singular}
Let $X=\coprod_{\tau\in T} \tau^\circ$ be a $\triang$-complex, $A$ a
subcomplex and $\shF$ be a $T$-acyclic sheaf on $X$.  For any
$i$, the natural map 
\[
H_i^T(X,A;\shF)\lra H_i(X,A;\shF)
\] is an isomorphism.
\end{theorem}

\begin{proof} 
There is also a natural map $H_i^{T^\bary}(X,A;\shF)\ra
H_i(X,A;\shF)$ and by Lemma~\ref{reduce-to-barycentric}, it suffices
to prove that this is an isomorphism. Moreover, by the long exact
sequences of homology of a pair, it suffice to prove the absolute
case, so assume $A=\emptyset$.  By the long exact sequences in
homology associated to the rows in the diagram in
Lemma~\ref{reducetoGr}, the five-Lemma and induction over $k$ (noting $\Gr^m_\shF=\shF^m$ for $m=\max_{\tau\in T}\dim\tau$), it suffices to prove that
the universal map
\begin{equation}
\label{eq-Cs}
C_\bullet^{T^\bary}(\Gr^k_\shF)\lra \overrightarrow{C}_\bullet(\Gr^k_\shF)
\end{equation}
induces an isomorphism in homology for each $k$. 
It suffices to verify the isomorphism for each summand on the right of \eqref{eq-Gr-decompose} individually, so fix
some $k$-simplex $\sigma\in T$ and let $v_\sigma$ denote the
barycenter of $\sigma$. In order to prove \eqref{eq-Cs} is an isomorphism, we are going to show that
\begin{equation}
\label{eq-Cs-local}
C_\bullet^{T^\bary}((i_k)_*(i_{\sigma^\circ})_!(\shF|_{\sigma^\circ}))\lra
\overrightarrow{C}_\bullet((i_k)_*(i_{\sigma^\circ})_!(\shF|_{\sigma^\circ}))
\end{equation}
is one.
We define the open star of $v_\sigma$, a
contractible open neighborhood of the interior of $\sigma$ in $X$, by
\[
U=\coprod_{{\tau\in T^\bary}\atop{v_\sigma\in\tau}} \tau^\circ.
\]
Let $M$ be the stalk of $\shF$ at a point in $\sigma^\circ$. 
By abuse of notation, we also denote by $M$ the constant sheaf with stalk $M$ on $U$.

\noindent\underline{Claim:} The right-hand side of \eqref{eq-Cs-local} naturally identifies with
$\overrightarrow{C}_\bullet(U;M)/ \overrightarrow{C}_\bullet (U\setminus\sigma^\circ;M)$.

\noindent\emph{Proof of Claim.}
Set $\shG:=(i_k)_*(i_{\sigma^\circ})_!(\shF|_{\sigma^\circ})$.
An element of $\overrightarrow{C}_\bullet(\shG)$ is an equivalence class represented by a finite sum 
$\sum_{j\colon \tau\ra X} \Gamma(\tau,j^{-1}\shG)$
under the equivalence given by restricting to barycentric refinements. 
Note that by the presence of $(i_{\sigma^\circ})_!$ in the definition of $\shG$, a summand $j\colon \tau\ra X$ contributes zero if its image meets $\partial\sigma$. 
This fact together with the fact that $U$ is an open neighborhood of $\sigma^\circ$ in $X$ implies that for every $j\colon \tau\ra X$ there is a barycentric refinement so that each simplex in the refinement that contributes non-trivially maps into $U$. We thus showed that the natural map 
$\overrightarrow{C}_\bullet(U;\shG)\ra \overrightarrow{C}_\bullet(X;\shG)$ is an isomorphism.

The sheaf $\shG$ restricted to $U$ simplifies to $\shG|_U=\iota_*(M|_{\sigma^\circ})$ for
$\iota\colon \sigma^\circ\ra U$ the embedding. 
After suitable finite refinement via the equivalence relation, a singular simplex $j\colon \tau\ra U$ becomes a chain $\sum_i {j_i}\colon \tau_i\ra U$ such that 
for each $j_i$ either $j_i^{-1}(\sigma^\circ)=\emptyset$ and thus $\Gamma(\tau_{i},j_i^{-1}\shG)=0$
or $j_i^{-1}(\sigma^\circ)$ is contractible and thus $\Gamma(\tau_{i},j_i^{-1}\shG)=M$.
The hereby defined map $\overrightarrow{C}_\bullet(U;\shG)\ra \overrightarrow{C}_\bullet(U;M)/ \overrightarrow{C}_\bullet (U\setminus\sigma^\circ;M)$
therefore yields an isomorphism.
\qed

Now, $\overrightarrow{C}_\bullet(U;M)/ \overrightarrow{C}_\bullet (U\setminus\sigma^\circ;M)$ computes the
singular homology $H_\bullet(U,U\setminus\sigma^\circ;M)$.  Most
importantly, we have reduced the situation to singular homology with
constant coefficients, so we are allowed to apply standard
techniques like deformation equivalences as follows. The pair
$(U,U\setminus\sigma^\circ)$ retracts to $(V,V\setminus v_\sigma)$
where  $V=\coprod_{{\tau\in T^\bary,v_\sigma\in\tau}
\atop{\tau\cap\partial\sigma=\emptyset}} \tau^\circ$.  By excision,
we transition to the pair $(\overline{V},\overline{V} \setminus
v_\sigma)$ where $\overline{V}$ is the closure of $V$ in $X$.  On
the other hand, $\overline{V}\setminus v_\sigma$ retracts to
$\overline{V}\setminus V$ inside $\overline{V}$. Summarizing, we
obtain isomorphisms
\[
H_\bullet(U,U\setminus\sigma^\circ;{M}) = H_\bullet(V,V\setminus
v_\sigma;{M}) = H_\bullet(\overline{V},\overline{V}\setminus
v_\sigma;{M}) = H_\bullet(\overline{V},\overline{V}\setminus V;{M}).
\]
On the other hand, we identify the left-hand side of \eqref{eq-Cs-local} as
\[
\bigoplus_{k\ge 0}\bigoplus_{{\tau\in
T^\bary,\dim\tau=k}\atop{v_\sigma\in\tau,
\tau\cap\partial\sigma=\emptyset}} M
\]
which coincides with $C_\bullet^{T^\bary\cap \overline V}(\overline
V; M)/ C_\bullet^{T^\bary\cap  (\overline V\setminus V)}(\overline
V\setminus V; M)$  noting that $\overline V$ and $\overline
V\setminus V$ are $\triang$-sub-complexes of $X$. The theorem now follows
from the known isomorphism of simplicial and singular homology for
constant coefficients
\[
H^{T^\bary\cap \overline
V}_\bullet(\overline{V},\overline{V}\setminus
V;{M})=H_\bullet(\overline{V},\overline{V}\setminus V;{M}),
\]
see for example \cite[Theorem 2.27]{Ha02}.
\end{proof}

%===========================================================
%
%		appendix B: Homology-cohomology isomorphism for  topological manifolds
%
%===========================================================

\section{Poincar\'e--Lefschetz isomorphism for constructible sheaves on topological manifolds}
\label{sec-hocoho-compare}
From now on, we use the notation $B,\P$ instead of $X,T$. The setup for the entire section is the following.
\begin{setup} 
\begin{enumerate}
\item
Let $B=\bigcup_{\tau\in\P} \tau^\circ$ be a $\triang$-complex with $\P$ \emph{finite} and $\Lambda$ a $\P$-constructible
sheaf on $B$. 
We assume there is no self-intersection of cells in $B$, i.e.~each $j_\tau$ is injective. 
We assume furthermore that the intersection of two cells in $\P$ is again a cell in $\P$.
We assume that for every $\tau\in\P$ either $\tau\cap\partial B=\emptyset$ or $\tau\cap\partial B$ is a face of $\tau$.
\item
We require that $B$ is an oriented topological manifold possibly with
non-empty boundary $\partial B$ which is also a topological manifold. We set $n=\dim B$.
\item
We denote by $\partial\P\subset\P$ the subcomplex given by the boundary of $B$ and define $\P^\max=\{\tau\in\P | \dim\tau=n\}$ and $\partial\P^\max=\{\tau\in\partial\P | \dim\tau=n-1\}$.
By the last assumption in (1), there is an injection $\partial\P^\max\ra\P^\max, \tau\mapsto\hat\tau$ with $\hat\tau$ such that $\hat\tau\cap\partial B=\tau$.
For $\tau\in\P^\max$, let $U_\tau$ denote a small open neighborhood of $\tau$ in $B$ and set $\fou=\{U_\tau|\tau\in \P^\max\}$.
Note that $U_{\tau_1}\cap\ldots\cap U_{\tau_k}$ is a small open
neighborhood of ${\tau_1}\cap\ldots\cap {\tau_k}$.
\item
We assume that $\Lambda$ is a $\P$-acyclic sheaf on $B$, i.e.~$H^i(\tau,\Lambda)=0$ for all $\tau\in\P$ and $i>0$ and in particular
\[
H^i(U_{\tau_1}\cap\ldots\cap U_{\tau_k},\Lambda)=0
\]
for $i>0$ and any subset $\{\tau_1,\ldots,\tau_k\}\subseteq \P^\max$.
\end{enumerate}
\label{setup-acyclic}
\end{setup}

\begin{example} 
If $\P'$ is a finite polyhedral complex that glues to an oriented
topological manifold $B$ with boundary and $\Lambda$ is a $\P'$-constructible
sheaf then the barycentric subdivision $\P$ of $\P'$ satisfies the
conditions of Setup~\ref{setup-acyclic} by Lemma~\ref{baryacyclic}. 
%Indeed (3) is deduced via Lemma~\ref{quotient-is-poset} and
%Lemma~\ref{poset-acyclic} as every barycentric simplex contains a
%unique minimal barycenter with respect to the ordering of $\P'$.
\end{example}

Fixing an orientation of each $\tau\in \P$, we can define the chain
complex $C^\P_\bullet(B,\partial B;\Lambda)$ as in
Definition~\ref{def-homology}, in particular 
\begin{equation}
\label{eq-def-homology}
C^\P_i(B,\partial B;\Lambda)  =
\bigoplus_{{\dim\tau=i}\atop{\tau\not\subset\partial
B}}\Gamma(\tau,\Lambda).
\end{equation}
To keep notation simple and since $H^\P_i(B,\partial B;\Lambda)=
H_i(B,\partial B;\Lambda)$ by Theorem~\ref{simplicial=singular},  we
denote $H^\P_i(B,\partial B;\Lambda)$ by $H_i(B,\partial B;\Lambda)$
and also $C^\P_i(B,\partial B;\Lambda)$ by $C_i(B,\partial
B;\Lambda)$, etc.
We have an exact sequence of chain complexes
\begin{equation}
\label{homology-ex-seq}
0\ra C_i(\partial B;\Lambda)\ra  C_i(B;\Lambda) \ra C_i(B,\partial B;\Lambda) \ra 0.
\end{equation}

The next goal is to produce a matching sequence of \v{C}ech co-chains.
We denote by $C^\bullet(B,\Lambda):=\check C^\bullet(\fou,\Lambda)$ the \v{C}ech
complex for $\Lambda$ with respect to $\fou$ and some total ordering of $\P^\max$. 
We have two induced covers for $\partial B$, namely $\fou^\partial:=\{U_{\hat\tau}\cap \partial B|\tau\in\partial\P^\max\}$ and
$\fou|_{\partial\P}:=\{U_{\tau}\cap\partial B|\tau\in\P\}\setminus\{\emptyset\}$ and $\fou^\partial\subseteq \fou|_{\partial\P}$.
The natural injection $\check C^\bullet(\fou^\partial,\Lambda)\ra \check C^\bullet(\fou|_{\partial\P},\Lambda)$ is a quasi-isomorphism.
We define $C^\bullet(\partial B,\Lambda):=\check C^\bullet(\fou|_{\partial\P},\Lambda)$ and
\begin{equation}
\label{def-coho-cone}
C^\bullet(B,\partial B;\Lambda):=\op{cone}\left( \check C^\bullet(\fou,\Lambda)\ra \check C^\bullet(\fou|_{\partial\P},\Lambda) \right),
\end{equation}
so we have a short exact sequence of complexes
\begin{equation}
\label{cohomology-ex-seq}
0\ra C^{j-1}(\partial B;\Lambda)\ra  C^j(B,\partial B;\Lambda) \ra C^j(B,\Lambda) \ra 0.
\end{equation}
and we next aim at producing a quasi-isomorphism from \eqref{homology-ex-seq} to \eqref{cohomology-ex-seq}.
Given
$I=\{\tau_0,\ldots,\tau_i\}\subset \P^\max$, we use the notation
$U_I=U_{\tau_0}\cap \ldots\cap U_{\tau_i}$.
We call $\P$ co-simplicial if the intersection of any set of $k+1$
many maximal cells is either empty or a $(n-k)$-dimensional simplex.
In the co-simplicial case the index sets of the sum of
\eqref{eq-def-homology} and of $C^j(B,\Lambda)$ are naturally in bijection (ignoring empty
$U_I$) and a map can be defined as the
isomorphism of complexes that is on each term given by
\begin{equation}
\label{co-simplicial-case}
\Gamma(\tau_0\cap\ldots\cap\tau_i,\Lambda)
\stackrel{\sim}{\longrightarrow} \Gamma(U_{\tau_0}\cap\ldots\cap
U_{\tau_i},\Lambda)
\end{equation}
up to some sign convention. We are going to show (see Proposition~\ref{prop-map-ho-coho}) that the 
map \eqref{co-simplicial-case} can be generalized to the situation where $\P$ is not necessarily co-simplicial by
replacing the right-hand side of \eqref{co-simplicial-case} by a
complex $C^\bullet_{\tau_0\cap\ldots\cap\tau_i}(B,\Lambda)$.  
This goes similar for the other two complexes.
Each non-empty $U_I$ has a unique cell $\tau$ in $\P$ that is
maximal with the property of being contained in it.  Fixing this
cell $\tau$, gathering all terms in the \v{C}ech complex for open
sets $U_I$ where $\tau$ is this unique maximal cell, yields a
subgroup $C^\bullet_\tau(B,\Lambda)\subseteq C^\bullet(B,\Lambda)$ and similarly for the other terms in \eqref{cohomology-ex-seq}. Precisely,
\[
C^i(B,\Lambda)=\bigoplus_{\tau\in\P} C^i_\tau(B,\Lambda) \qquad
\hbox{ for }  \qquad C^i_\tau(B,\Lambda) := \bigoplus_{\left\{
I\left| {I\subseteq\P^\max, |I|=i+1, \tau\subset U_I}\atop{\sigma\not\subset
U_I\hbox{\tiny whenever }\tau\subsetneq\sigma}\right.\right\}}
\Gamma(U_I,\Lambda),
\]
\[
C^i(\partial B,\Lambda)=\bigoplus_{\tau\in\partial\P} C^i_\tau(\partial B,\Lambda) \quad
\hbox{for}  \quad C^i_\tau(\partial B,\Lambda) := \bigoplus_{\left\{
I\left| {I\subseteq\P^\max, |I|=i+1, \tau\subset U_I}\atop{\sigma\not\subset
U_I\hbox{\tiny whenever }\tau\subsetneq\sigma\in\partial\P}\right.\right\}}
\Gamma(U_I\cap \partial B,\Lambda),
\]
\[
C^i(B,\partial B;\Lambda)=\bigoplus_{\tau\in\P} C^i_\tau(B,\partial B;\Lambda) \qquad
\hbox{for}  \qquad C^i_\tau(B,\partial B;\Lambda) := C^i_\tau(B;\Lambda)\,\oplus\, C^{i-1}_\tau(\partial B,\Lambda).
\]
We consider decreasing filtrations $F^\bullet$ by sub-complexes of the three complexes in \eqref{cohomology-ex-seq} respectively so the maps in the sequence respect the filtrations, define $F^kC^i(B,\Lambda) = \bigoplus_{{\tau\in\P}\atop{n-\dim\tau\ge k}} C^i_\tau(B,\Lambda)$ and by the same formula also for the other two.
Indeed, one checks that the maps in \eqref{cohomology-ex-seq} become strict for these filtrations, i.e.~the induced sequence on each graded piece is also exact.
Indeed, this follows from
\[
\Gr_F^kC^i(B,\Lambda) = F^kC^i(B,\Lambda)/F^{k+1}C^i(B,\Lambda)=\bigoplus_{{\tau\in\P}\atop{n-\dim\tau=k}}C^i_\tau(B,\Lambda)
\]
and a similar expression for the other two complexes in \eqref{cohomology-ex-seq}.
\begin{lemma} 
\label{lemma-gradedcoho}
We have $H^i(\Gr_F^kC^\bullet(B,\Lambda))=H^i(\Gr_F^kC^\bullet(\partial B,\Lambda))=H^i(\Gr_F^kC^\bullet(B,\partial;\Lambda))=0$ unless $i=k$ and then
\[
H^i(\Gr_F^iC^\bullet(B,\Lambda))
=\bigoplus_{{\tau\in\P,\tau\not\subset\partial B}\atop{n-\dim\tau=i}}
\Gamma(\tau,\Lambda),
\qquad
H^i(\Gr_F^iC^\bullet(\partial B,\Lambda))
=\bigoplus_{{\tau\in\P,\tau\subset\partial B}\atop{n-\dim\tau=i}}
\Gamma(\tau,\Lambda)
\]
and $H^i(\Gr_F^iC^\bullet(B,\partial B;\Lambda))=\bigoplus_{\tau\in\P,\,n-\dim\tau=i}
\Gamma(\tau,\Lambda)$.
\end{lemma}
\begin{proof} 
We first deal with $C^\bullet(B,\Lambda)$. Observe that
$
\Gr_F^kC^\bullet(B,\Lambda)=\bigoplus_{{\tau\in\P}\atop{\codim\tau=k}}
C_\tau^\bullet(B,\ZZ)\otimes_\ZZ \Gamma(\tau,\Lambda),
$
so to deal with $C^\bullet(B,\Lambda)$, it suffices to show that $H^i(C_\tau^\bullet(B,\ZZ))$ is isomorphic
to $\ZZ$ when $\codim\tau=i$ and trivial otherwise.  The set
$\fou_\tau=\{U_\sigma\in\fou| \tau\subset\sigma \}$ covers
an open ball containing $\tau$.  Let $C^\bullet(\fou_\tau,\ZZ)$
denote the associated \v{C}ech complex. We have a short exact
sequence of complexes
\begin{equation}
\label{Utau-ses}
0\lra  C_\tau^\bullet(B,\ZZ) \lra C^\bullet(\fou_\tau,\ZZ)\lra
\overline{C}_\tau^\bullet(\ZZ) \lra 0
\end{equation}
where 
$
\overline{C}_\tau^i(\ZZ) =
\bigoplus \Gamma(U_{\sigma_0}\cap\ldots\cap U_{\sigma_i},\ZZ)
$
is the induced cokernel that has its sum running over 
subsets $\{U_{\sigma_0},\ldots,U_{\sigma_i}\}\subset\fou_\tau$ with
$\tau\neq\sigma_0\cap\ldots\cap\sigma_i$. Denoting
$K=\bigcup_{\sigma\in\P,\tau\subseteq\sigma}\sigma$, one finds the
sequence \eqref{Utau-ses} is naturally identified with a sequence of
\v{C}ech complexes computing the long exact sequence
\[
\ldots\lra H^i_\tau(K,\ZZ)\lra H^i(K,\ZZ)\lra
H^i(K\setminus\tau,\ZZ) \lra\ldots
\]
and we have
\[
K\setminus\tau\hbox{ is }
\left\{
\begin{array}{l} 
\hbox{homotopic to }S^d\hbox{ with }d=\codim\tau-1\hbox{ if
}\tau\not\subset\partial B,\\ \hbox{contractible if
}\tau\subset\partial B.
\end{array}\right.
\]
Since $K$ is contractible we get that $H^i_\tau(K,\ZZ)=0$ for all
$i$ if $\tau\subset\partial B$, that is, $C^\bullet_\tau(B,\ZZ)$ is
exact in this case. Otherwise, we find $H^i_\tau(K,\ZZ)\cong\ZZ$ for
$i=\codim\tau$ and trivial otherwise. The choice of the isomorphism
depends on the orientation of $S^d$ which can be taken to be the
induced one from the orientations of $B$ and $\tau$.

The same arguments give the claim also for $C^\bullet(\partial B,\Lambda)$ after noting that the quasi-isomorphism
$\check C^\bullet(\fou^\partial,\Lambda)\ra \check C^\bullet(\fou|_{\partial\P},\Lambda)$ is strictly compatible with the filtration $F$.

Finally the claim for $C^\bullet(B,\partial B;\Lambda)$ follows because this complex can be filtered with two graded pieces that are the complexes that we just dealt with.
\end{proof}

\begin{lemma} 
\label{spec-seq-ho-coho}
The spectral sequence
$$
E_1^{p,q}=H^{p+q}(\Gr^p_F C^\bullet(B,\Lambda)) \Rightarrow
H^{p+q}(B,\Lambda).
$$
degenerates at $E_2$ and a similar statement holds for $C^\bullet(\partial B,\Lambda)$ and $C^\bullet(B,\partial B;\Lambda)$.
\end{lemma}
\begin{proof} The $E_1$ page is concentrated in
$q=0$ by Lemma~\ref{lemma-gradedcoho}. 
\end{proof}
Let $d_1^{p,q}\colon E_1^{p,q}\ra E_1^{p+1,q}$ denote the differential of the $E_1$ page.

\begin{proposition} 
\label{prop-map-ho-coho}
By Lemma~\ref{lemma-gradedcoho}, we have $H^{n-i}(\Gr_F^{n-i}
C^\bullet(B,\Lambda)) =
\bigoplus_{{\tau\in\P,\tau\not\subset\partial B}\atop{\dim\tau=i}}
\Gamma(\tau,\Lambda)$ and therefore an identification
$
f\colon C_i(B,\partial B;\Lambda)\lra H^{n-i}(\Gr^{n-i}_F C^{\bullet}(B;\Lambda)).
$
which fits in an isomorphism of exact sequences of complexes
$$
\resizebox{\textwidth}{!}{
\xymatrix@C=30pt
{ 
0\lra C_i(\partial B;\Lambda)\ar[r]\ar[d] & C_i(B;\Lambda) \ar[r]\ar[d] & C_i(B,\partial B;\Lambda)\ar^f[d] \lra 0\\
0\ra  H^{n-1-i}(\Gr^{n-1-i}_F C^{\bullet}(\partial B;\Lambda))\ar[r] &  H^{n-i}(\Gr^{n-i}_F C^{\bullet}(B,\partial B;\Lambda)) \ar[r]& H^{n-i}(\Gr^{n-i}_F C^{\bullet}(B;\Lambda))\ra  0.
}
}
$$
i.e.~the map $f$ upgrades to an isomorphism of complexes (varying $i$) when taking
$\partial_i$ and $d^{n-i,0}_1$ for the differentials respectively and a similar statement holds for the other vertical maps in the diagram.
\end{proposition}

\begin{proof}  
We need to show that $f$ commutes with differentials, i.e.~that
$d^{\bullet,\bullet}_{1}f=f\partial_\bullet$.  It will be sufficient
to show for an $i$-simplex $\tau$ and a facet $\omega$ of $\tau$
with $\omega\not\subset\partial B$ that for every element
$\alpha\in\Gamma(\tau,\Lambda)$, we have
\[
f((\partial_i\alpha)_\omega)=(d_1^{n-i,0}f \alpha)_\omega.
\]
where $(\partial_i\alpha)_\omega$ denotes the projection of
$\partial_i\alpha$ to ${\Gamma(\omega,\Lambda)}$ and similarly
$(d_1^{n-i,0}f \alpha)_\omega$ denotes the projection of
$d_1^{n-i,0}f \alpha$ to ${H^{n-i+1}C_\omega^{\bullet}(B,\Lambda)}$.
We first do the case $\Lambda=\ZZ$. 
The map
$$
d_1^{n-i,0}\colon  H^{n-i}C_\tau^{\bullet}(B,\ZZ) \lra
{H^{n-i+1}C_\omega^{\bullet}(B,\ZZ)} 
$$
is the composition of the \v{C}ech differential with projection:
\begin{equation} \label{eq-compos}
\bigoplus_{\left\{ I \left| {|I|=i+1, \tau\subset
U_I}\atop{\sigma\not\subset U_I\hbox{\tiny whenever
}\tau\subsetneq\sigma}\right.\right\}} \Gamma(U_I,\ZZ)
\lra
\bigoplus_{\left\{ I \left| |I|=i+2\right.\right\}} \Gamma(U_I,\ZZ) 
\lra
\bigoplus_{\left\{ I \left| {|I|=i+2, \omega\subset
U_I}\atop{\sigma\not\subset U_I\hbox{\tiny whenever
}\omega\subsetneq\sigma}\right.\right\}} \Gamma(U_I,\ZZ).
\end{equation}
\begin{minipage}[b]{0.7\textwidth}
To better understand it, let $B_\sigma$ denote a suitably
embedded closed ball of dimension $n-\dim\sigma$ in $B$ meeting
$\sigma$ transversely (in a point) where $\sigma$ stands for $\tau$ or
$\omega$. Without loss of generality, $B_\tau$ is part of the boundary of $B_\omega$, see the 
illustration on the right.  

Observe that, similarly to the proof of Lemma~\ref{lemma-gradedcoho},
the \v{C}ech complex $C^\bullet_\sigma(B,\ZZ)$ naturally computes
$H^{\bullet}_{B^\circ_\sigma}(B_\sigma,\ZZ)$
where $B_\sigma^\circ$
is a closed ball inside $B_\sigma$ obtained by a small shrinking of
\end{minipage}
%second column
\begin{minipage}[b]{0.3\textwidth}
\qquad \includegraphics{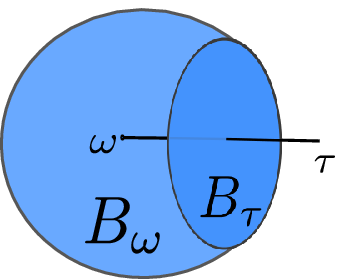}\\
\end{minipage}
 $B_\sigma$. We claim that \eqref{eq-compos} is given by the composition of natural maps
\begin{equation} \label{eq-compos2}
H^{n-i}_{B^\circ_\tau}(B_\tau,\ZZ)\lra
H^{n-i}_{B^\circ_\tau}(\partial B_\omega,\ZZ) \lra H^{n-i}(\partial
B_\omega,\ZZ)\lra H^{n-(i-1)}_{B^\circ_\omega}(B_\omega,\ZZ).
\end{equation}
Indeed, we may replace the middle term in \eqref{eq-compos} by
$
\bigoplus_{\left\{ I \left| \omega\subset U_I,
|I|=i+1\right.\right\}} \Gamma(U_I,\ZZ_{\partial B_\omega})
$ 
without changing the composition. This identifies \eqref{eq-compos} and \eqref{eq-compos2}. 
As an elementary fact, each map in \eqref{eq-compos2} is an isomorphism if $n-i>0$.
If $n=i$, i.e.~$\dim B_\tau=0$, the third term has rank two but the composition is still an isomorphism.
If $\omega$ has the induced orientation from $\tau$ then the orientation of $B_\tau$ is
the induced one from $\partial B_\omega$ so there is the sign agrees with that of the map
$\Gamma(\tau,\ZZ)\ra \Gamma(\omega,\ZZ)$ in $C_\bullet(B,\ZZ)$. We finished showing the $\Lambda=\ZZ$ case.

The general case follows directly as the component of the differentials
we considered is then just additionally tensored with the
restriction map $\Gamma(\tau,\Lambda)\ra \Gamma(\omega,\Lambda)$ in
the source as well as in the target of $f$ as observed before in the proof of Lemma~\ref{lemma-gradedcoho}.

The proof for the other vertical maps in the diagram follows from that for $f$: indeed, the left vertical map can be treated on the same footing and the middle one is then composed from the other two maps.
\end{proof}

\begin{theorem}
\label{cor-iso-ho-coho}
Given Setup~\ref{setup-acyclic}, the diagram in Proposition~\ref{prop-map-ho-coho} induces a natural isomorphism of long exact sequences 
\begin{equation*}
\resizebox{\textwidth}{!}{
\xymatrix@C=30pt
{ 
\dots\ar[r]& H_k(\partial B) \ar[r]\ar^\sim[d] & H_k( B) \ar[r]\ar^\sim[d] & H_k(B,\partial B) \ar[r]\ar^\sim[d]& H_{k-1}(\partial B) \ar[r]\ar^\sim[d] &\dots\\
\dots\ar[r]& H^{(n-1)-k}(\partial B) \ar[r] & H^{n-k}( B,\partial B) \ar[r] & H^{n-k}(B) \ar[r] &H^{(n-1)-(k-1)}(\partial B) \ar[r] &\dots
}
}
\end{equation*}
where all coefficients are $\Lambda$.
\end{theorem}

\begin{proof} 
By Proposition~\ref{prop-map-ho-coho} and Lemma~\ref{spec-seq-ho-coho} we obtain an isomorphism
$\Gr^\bullet_F H_i(B,\partial B;\Lambda)\ra \Gr_F^\bullet H^{n-i}(B;\Lambda)$
where the filtration $F$ on homology is defined in the
straightforward manner.  The other vertical maps in 
the diagram of Proposition~\ref{prop-map-ho-coho} may be given a similar interpretation.
We may then remove $\Gr_F^\bullet$ from this map
because the graded pieces are concentrated in a single degree
by Lemma~\ref{lemma-gradedcoho}.
\end{proof}

\begin{corollary} 
\label{PLmap-is-the-same}
The vertical maps in Theorem~\ref{cor-iso-ho-coho} commute with maps of the sheaves of coefficients and are compatible with long exact sequences of (co)homology resulting from short exact sequences of $\P$-constructible sheaves.
Furthermore, the vertical maps agree with the inverses of the maps given in \cite[Theorem~V-9.3, p.330]{Br97} when inserting the given orientation to have an isomorphism $\shO\cong j_!\ZZ$ with $j\colon B\setminus\partial B\hra B$ the inclusion. 
In particular, the maps in \cite{Br97} are isomorphisms for $\P$-constructible sheaves on $B$.
\end{corollary}
\begin{proof} 
By Lemma~\ref{reduce-to-barycentric}, the barycentric refinement is an isomorphism on homology. 
By Lemma~\ref{baryacyclic}, a $\P$-constructible sheaf is $\P^\bary$-acyclic.
This allows us to produce the vertical maps in Theorem~\ref{cor-iso-ho-coho} for all $\P$-constructible sheaves, not just $\P$-acyclic ones.
The first statement follows from the observation that a short exact sequence of $\P$-acyclic sheaves yields a short exact sequence of diagrams of the type of the diagram given in Proposition~\ref{prop-map-ho-coho}. 

For the second statement, note that we may view $\P$ as a topological space by declaring a set $U$ to be open if for every $\tau\in U$ we have $\tau'\in U$ whenever $\tau\subset\tau'$. 
The map $\pi\colon B\ra \P$ that sends a point in the relative interior of $\tau$ to $\tau$ is continuous and open and induces a natural equivalence of sheaves on $\P$ and $\P$-constructible sheaves on $B$. 
We are going to apply \cite[Theorem~II-6.2, p.53]{Br97}, i.e.~say that the maps agree because both sides are cohomological functors. 
The cohomological functors we consider are functors of sheaves $\shF$ on $\P$. 
We only do the maps $H_i(B,\partial B;\Lambda)\ra H^{n-i}(B;\Lambda)$ because the one for $\partial B$ goes similar and then the five lemma implies the assertion.
The first functor to consider is $F^i(\shF)=H_{n-i}(B,\pi^*\shF)$ and the second is $G^i(\shF)=H^i(B,\partial B;\pi^*\shF)$.
The base case $F^0(\shF)\ra G^0(\shF)$ is the natural isomorphism
$H_n(B,\pi^*\shF)\ra H^0(B,\partial B;\pi^*\shF)\cong H^0(B,\shO\otimes\pi^*\shF)$ which does coincide with the map in \cite{Br97}. 
Using Theorem~\ref{simplicial=singular} and \cite{Br97}, one checks that $F^i$ and $G^i$ satisfy the required properties in \cite[before Theorem~II-6.2]{Br97} so the cited uniqueness theorem implies the assertion.
\end{proof}

%===========================================================

		% end the bibliography

\end{document}